\documentclass[a4paper,12pt,reqno]{amsart}
\usepackage{amsmath}
\usepackage{amsfonts}
\usepackage{amssymb}
\usepackage{amsthm}
\usepackage{amscd}
\usepackage{xcolor}
\usepackage{verbatim}
\usepackage{eucal,url,amssymb,stmaryrd,enumerate,amscd,}
\usepackage{amsfonts}

\usepackage{amsmath,amsthm,amssymb,amscd,enumerate,eucal,url,stmaryrd}
\usepackage{verbatim}

\newcommand{\la}{\langle}
\newcommand{\ra}{\rangle}

\newcommand{\SA}{{\mathcal{A}}}

\newcommand{\ZZ}{\mathbb{Z}}
\newcommand{\CC}{\mathbb{C}}

\newcommand{\RR}{\mathbb{R}}
\newcommand{\SO}{\mathrm{SO}}
\newcommand{\SU}{\mathrm{SU}}
\newcommand{\U}{\mathrm{U}}
\newcommand{\RP}{\mathbb{RP}}
\newcommand{\id}{\Id}
\newcommand{\ox}{\otimes}
\DeclareMathOperator{\tr}{tr}
\DeclareMathOperator{\Id}{Id}
\DeclareMathOperator{\End}{End}

\numberwithin{equation}{section}

\theoremstyle{plain}
\newtheorem{proposition}{Proposition}[section]
\newtheorem{theorem}[proposition]{Theorem}
\newtheorem{lemma}[proposition]{Lemma}

\theoremstyle{definition}
\newtheorem{definition}[proposition]{Definition}

\theoremstyle{remark}
\newtheorem{remark}[proposition]{Remark}

\begin{document}

\title[On ${\mathrm{SO}}(3)$-bundles over the Wolf spaces]{On $\SO(3)$-bundles over the Wolf spaces}

\author[M. Fern\'andez]{Marisa Fern\'{a}ndez}
\address{Universidad del Pa\'{\i}s Vasco,
Facultad de Ciencia y Tecnolog\'{\i}a, Departamento de Mate\-m\'a\-ticas,
Apartado 644, 48080 Bilbao, Spain}
\email{marisa.fernandez@ehu.es}

\author[V. Mu\~{n}oz]{Vicente Mu\~{n}oz}
\address{Facultad de Ciencias Matem\'aticas, Universidad
Complutense de Madrid, Plaza de Ciencias
3, 28040 Madrid, Spain}
\email{vicente.munoz@mat.ucm.es}

\author[J. S\'anchez]{Jonatan S\'anchez}
\address{Universidad del Pa\'{\i}s Vasco,
Facultad de Ciencia y Tecnolog\'{\i}a, Departamento de Mate\-m\'a\-ticas,
Apartado 644, 48080 Bilbao, Spain}
\email{jonatan.sanchez@ehu.es}

\subjclass[2010]{Primary 53C25, 55P62, 57N65 Secondary 55S30, 53C26}

\keywords{$3$-Sasakian homogeneous spaces, $3$-sphere bundles, Wolf spaces, formality, Massey products.}

\begin{abstract}
We study the formality of the total space of principal $\SU(2)$ and $\SO(3)$-bundles over a Wolf space,
that is a symmetric positive quaternionic K\"ahler manifold. We apply this to conclude that all  the
$3$-Sasakian homogeneous spaces are formal.  We also determine the
principal $\SU(2)$ and $\SO(3)$-bundles over the Wolf spaces whose total space is non-formal.
\end{abstract}

\maketitle

\section{Introduction}\label{sec:intro}

A Riemannian manifold $(S, g)$ is called {\em $3$-Sasakian manifold} if $S\times{\mathbb{R}}^+$ equipped 
with the cone metric $g^c =  t^2 g+dt^2$ is hyperk\"ahler, and so the holonomy group of $g^c$ is a subgroup of
$\mathrm{Sp}(n+1)$, where $2n+2$ is the complex dimension of the hyperk\"ahler cone, and so
$S$ has odd dimension $4n+3$.
The hyperk\"ahler structure on the cone $(S\times{\mathbb{R}}^+, g^c =  t^2 g\,+\,dt^2)$ induces 
a {\em $3$-Sasakian structure} on the base of the cone. In particular, the triple of
complex structures on $S\times{\mathbb{R}}^+$ gives rise to a triple of Reeb vector fields
$(\xi_1,\xi_2, \xi_3)$ on $S$ whose Lie brackets give a copy of the Lie algebra $\mathfrak{su}(2)$ (see section \ref{$3$-sasaki:formal} for details). 
 
 A $3$-Sasakian manifold $(S, g)$ is said to be {\em regular} if the vector fields $\xi_1, \xi_2, \xi_3$ are complete and 
the corresponding $3$-dimensional foliation is regular, so that the space of leaves is a smooth 
$4n$-dimensional manifold $M$. Ishihara and Konishi \cite{Ishihara-Konishi} noticed that the induced metric on the 
latter is quaternionic K\"ahler with positive scalar curvature, that is $M$ is a {\em positive 
quaternionic K\"ahler manifold}. Conversely \cite{Konishi, Tanno}, starting with a 
positive quaternionic K\"ahler manifold $M$, 
the manifold $M$ can be recovered as the total space of a bundle naturally associated to $M$.

Salamon \cite{Salamon1} proved that compact positive quaternionic K\"ahler manifolds are simply connected 
and their odd Betti numbers are zero. 
Important results on the topology of a compact $3$-Sasakian manifold
were proved by Galicki and Salamon \cite{GS}, showing that the odd Betti numbers
$b_{2i+1}$ of such a manifold of dimension $4n+3$, are all zero for $0\leq i \leq n$.
Moreover, for regular compact $3$-Sasakian manifolds many topological properties are known
(see  \cite[Proposition 13.5.6 and Theorem 13.5.7]{BG}). For example, such a manifold is simply connected unless $N\,=\,{\mathbb{RP}}^{4n+3}$.
Also, using the results of LeBrun and Salamon \cite{LeBrun-Salamon} about the topology of positive quaternionic K\"ahler manifolds,
Boyer and Galicki \cite{BG} show interesting relations among the Betti numbers of regular compact $3$-Sasakian manifolds;
in particular that $b_{2}\,\leq\,1$.

In this paper we deal with homotopical properties of {\em 3-Sasakian homogeneous spaces}. Such a space $S$ is a 3-Sasakian manifold 
with a transitive action of the group of automorphisms of the Sasakian 3-structure (see section \ref{$3$-sasaki:formal} for details). 
The $3$-dimensional foliation on $S$ is regular, and the space of leaves is a homogeneous positive quaternionic K\"ahler manifold, that is a 
symmetric positive quaternionic K\"ahler manifold \cite{Alek}. Such quaternionic K\"ahler manifolds 
are given by  the infinite series ${\mathbb{H}}{\mathbb{P}}^{n}$, ${\mathbb{G}r}_2({\mathbb{C}}^{n+2})$ and 
$\widetilde{\mathbb{G}r}_4({\mathbb{R}}^{n+4})$ 
(the Grassmannian of oriented real 4-planes) and by the exceptional symmetric spaces of compact type
  \begin{align*}
  G I\,=\,\frac{{\mathrm{G}}_2}{{\mathrm{SO}}(4)}, \qquad F I\,=\,\frac{{\mathrm{F}}_4}{{\mathrm{Sp}}(3)\cdot{\mathrm{Sp}}(1)}, \qquad 
  E II\,=\,\frac{{\mathrm{E}}_6}{{\mathrm{SU}}(6)\cdot{\mathrm{Sp}}(1)}, \\  
  E VI\,=\, \frac{{\mathrm{E}}_7}{{\mathrm{Spin}}(12)\cdot{\mathrm{Sp}}(1)} \qquad  \text{and} 
  \qquad E IX\,=\, \frac{{\mathrm{E}}_8}{{\mathrm{E}}_7\cdot{\mathrm{Sp}}(1)},
  \end{align*}
which are part of the E. Cartan classification \cite{Cartan}.
The corresponding $3$-Sasakian homogeneous spaces are given in Theorem \ref{class-homog-3-sasaki}. 
With the exception of the sphere $S^{4n+3}$, which is an $\SU(2)={\mathrm{Sp}}(1)$-bundle over the quaternionic projective space
${\mathbb{H}}{\mathbb{P}}^{n}$, 3-Sasakian homogeneous spaces are principal $\SO(3)$-bundles over the spaces listed above. 
The Euler class $e(S)$ of the $\SO(3)$-bundle is $-\frac14 p_1(S)$, where 
$p_1(S)$ is the first Pontryagin class of the $\SO(3)$-bundle and it is defined by the
quaternionic K\"ahler form of the quaternionic K\"ahler manifold.

A simply connected manifold is formal if its rational homotopy type
is determined by its rational cohomology algebra. We shall say
that $M$ is {\it formal\/} if its minimal model is formal or, equivalently, if the de Rham
complex $(\Omega^*( M), d)$ of $M$ and the algebra of the de Rham cohomology
$(H^*(M), d=0)$ have the same minimal model (see section \ref{min-models} for details).

A celebrated result of Deligne, Griffiths, Morgan and Sullivan states that any compact K\"ahler manifold is formal \cite{DGMS}.
In the same spirit, formality
of a manifold is related to the existence of suitable geometric structures on the
manifold. Amann and Kapovitch \cite{Amann2} have proved that positive quaternionic K\"ahler manifolds are formal.
Many other examples of formal manifolds are known: spheres, projective spaces, compact Lie groups, symmetric spaces 
of compact type and flag manifolds. Nevertheless, there are examples of non-formal homogeneous
spaces (see~\cite{Amann3} and references therein).

For Sasakian manifolds, that is Riemannian manifolds whose cone metric is K\"ahler, the first and second authors 
have proved that the formality is not an obstruction
to the existence of Sasakian structures even on  simply connected manifolds \cite{BFMT}. 
However, in \cite{FIM} it is proved that the formality allows one to distinguish $7$-dimensional
Sasaki-Einstein manifolds which admit $3$-Sasakian structures from those which do not.

Since the aforementioned homogeneous quaternionic spaces are formal, it seems interesting to understand the formality not only of the 
$3$-Sasakian homogeneous spaces but also of the total space of any principal $\SU(2)=S^3$ and any principal $\SO(3)=\RP^3$-bundle over a
homogeneous positive quaternionic K\"ahler manifold.

For a fibration $F \rightarrow E \rightarrow  B$ there are conditions on 
the base $B$ and the fiber $F$ which imply that if $B$ is formal, then $E$ is formal  \cite{Amann2, Lupton}. 
Let us recall that a simply connected topological space $F$ is called {\em positively elliptic}, or
$F_{0}$, if it is rationally elliptic, that is if it has finite dimensional rational homotopy and cohomology,
and if it has positive Euler characteristic. 
In this case, its rational cohomology is concentrated in even degrees only.
For $F_{0}$-spaces, Halperin's conjecture states that if $F$ is such a space, then 
$H^*(F, \mathbb{Q})$ has no negative degree derivations. 
Lupton in \cite{Lupton} proved that if the base space $B$ of the fibration is simply connected, and the fiber $F$ is $F_{0}$ and 
satisfies Halperin's conjecture, then the formality of $B$ implies the formality of $E$ (see also \cite{Amann2}).
Lupton's result can not be applied to $\SO(3)$ nor $\SU(2)$-bundles because the mentioned conditions 
 on the fiber $F$ fail.

The structure of the paper is as follows. First, 
in sections \ref{$3$-sasaki:formal} and \ref{min-models} we review some definitions and results about $3$-Sasakian homogeneous spaces and 
 models (not necessarily minimal) of $\SO(3)$ and $\SU(2)$-fibrations. In particular, we recall the calculation of
the first Pontryagin class of an $\SO(3)$-bundle and the second Chern class of
an $\SU(2)$-bundle (see \eqref{eqn:p1}-\eqref{eqn:pont-chern}).

Section \ref{sect:grassmann1} is devoted to the study of the formality 
of $\SU(2)$ and $\SO(3)$-bundles over the complex Grassmannian ${\mathbb{G}r}_2({\mathbb{C}}^{n+2})$.
Using the concept of $s$-formal minimal model, 
introduced in \cite{FM} as an extension of formality \cite{DGMS}, we determine the principal
$\SU(2)$ and $\SO(3)$-bundles over ${\mathbb{G}r}_2({\mathbb{C}}^{n+2})$ whose total space is formal 
(Theorem \ref{th: 3-sasa-complexgr-1} and Theorem \ref{th: nonformal-complexgr-1}). In
 Theorem \ref{th: nonformal-complexgr-1} we determine
the principal $\SU(2)$ and $\SO(3)$-bundles over ${\mathbb{G}r}_2({\mathbb{C}}^{n+2})$ whose total space is non-formal 
because it has a non-trivial Massey product. In particular, if $n$ is even, we show that there are non-formal $\SU(2)$ and $\SO(3)$-bundles 
over ${\mathbb{G}r}_2({\mathbb{C}}^{n+2})$. The characterization of all these bundles is given by the Euler class of the bundle.
On the other hand, we show the cohomology class of the quaternionic K\"ahler form on ${\mathbb{G}r}_2({\mathbb{C}}^{n+2})$ 
in terms of the generators of the rational cohomology of ${\mathbb{G}r}_2({\mathbb{C}}^{n+2})$
(Proposition \ref{prop:formgrassc}). Then, from Theorem \ref{th: 3-sasa-complexgr-1} we conclude that
the $3$-Sasakian homogeneous space ${\mathrm{SU}}(n+2)/{\mathrm{S}}\big({\mathrm{U}}(n)\times{\mathrm{U}}(1)\big)$ is formal.

In section \ref{sect:grassmann2}, proceeding in the same way as in section \ref{sect:grassmann1}, we determine the principal
$\SU(2)$ and $\SO(3)$-bundles over the oriented real Grassmannian $\widetilde{\mathbb{G}r}_4({\mathbb{R}}^{n+4})$ 
whose total space is formal and also those 
whose total space is non-formal (Theorems \ref{thm:prop:GrRevenformal}, \ref{thm:prop:GrR8formal}
and \ref{thm:prop:formalidadreal-final}).
Nevertheless, this study is more subtle than the one made in section \ref{sect:grassmann1}.
It is due to the fact that the rational cohomology of $\widetilde{\mathbb{G}r}_4({\mathbb{R}}^{n+4})$ changes depending mainly of whether $n$ is even or odd.
In any case, we show that the $3$-Sasakian homogeneous space ${\mathrm{SO}}(n+4) / ({\mathrm{SO}}(n)\times {\mathrm{Sp}}(1))$
is formal. From this result, Theorem \ref{th: 3-sasa-complexgr-1} and the formality of 
the sphere $S^{4n+3}$ and the real projective space ${\mathbb{R}}{\mathbb{P}}^{4n+3}$, we get 
that {\em the non-exceptional $3$-Sasakian homogeneous spaces are formal}.
 The formality of the exceptional $3$-Sasakian homogeneous spaces and, more generally, the formality of the total space 
of $\SU(2)$ and $\SO(3)$-bundles over the exceptional Wolf spaces is proved in section \ref{sect:exceptional}.

A central point in the study of positive quaternionic K\"ahler manifolds is the LeBrun-Salamon conjecture
\cite{LeBrun-Salamon} that says that every positive quaternionic K\"ahler
manifold is a symmetric space. This has been proved by Hitchin in dimension $4$, by
Poon-Salamon in dimension $\leq 8$, and by Herrera and Herrera in dimension $12$.
By Theorem \ref{th:structure}, any compact regular $3$-Sasakian manifold $S$ is an $\SU(2)$ or $\SO(3)$-bundle
over a compact positive quaternionic K\"ahler manifold $M$.
Therefore, our results prove that if $S$ is a compact regular $3$-Sasakian 
manifold of dimension $\leq 15$, then $S$ is formal.

\section{Homogeneous $3$-Sasakian manifolds}\label{$3$-sasaki:formal}

We recall the notion of homogeneous $3$-Sasakian space and the classification theorem of these spaces
following \cite{Blair,BG,BG-2}.

An odd dimensional Riemannian manifold $(S,g)$ is Sasakian if its cone $(S\times{\mathbb{R}}^+, g^c = t^2 g+dt^2)$ is K\"ahler, that
is the cone metric $g^c = t^2 g+dt^2$ admits a compatible integrable almost complex structure $J$ so that
$(S\times{\mathbb{R}}^+, g^c = t^2 g+dt^2, J)$ is a K\"ahler
manifold. In this case the  Reeb vector field $\xi=J\partial_t$ is
a Killing vector field of unit length. The corresponding $1$-form
$\eta$ defined by $\eta(X)=g(\xi,X)$, for any vector field $X$ on
$S$, is a contact form. Let $\nabla$ be the Levi-Civita connection
of $g$. The (1,1) tensor $\phi X=\nabla_X\xi$ satisfies the
identities 
 $$
 \phi^2=-\id+\eta\otimes\xi, \quad g(\phi X,\phi Y)=g(X,Y)-\eta(X)\eta(Y), 
\quad d\eta(X,Y)=2g(\phi X,Y),
 $$
for vector fields $X,Y$ on $S$.

A collection of three Sasakian structures on a
$(4n+3)$-dimensional Riemannian manifold satisfying
quaternionic-like identities form a $3$-Sasakian structure. More
precisely, a Riemannian manifold $(S, g)$ of dimension $4n+3$ is
called $3$-Sasakian if its cone $(S\times{\mathbb{R}}^+, g^c = t^2
g\,+\,dt^2)$ is hyperk\"ahler, that is the metric $g^c = t^2
g\,+\,dt^2$ admits three compatible integrable almost complex
structure $J_s$, $s=1,2,3$, satisfying the quaternionic relations,
i.e.,\ $J_1J_2=-J_2J_1=J_3$, such that $(S\times{\mathbb{R}}^+,
g^c = t^2 g\,+\,dt^2, J_1, J_2, J_3)$ is a hyperk\"ahler manifold.
Equivalently, the holonomy group of the cone metric $g^c$ is a
subgroup of $\mathrm{Sp}(n+1)$. In this case the Reeb vector
fields $\xi_s=J_s\partial_t$, $s=1,2,3$, are Killing vector
fields. The  three Reeb vector fields $\xi_s$, the three
$1$-forms  $\eta_s$ and the three $(1,1)$ tensors $\phi_s$,
$s=1,2,3$, satisfy the relations
\begin{align*}
&\eta_i(\xi_j)=g(\xi_i,\xi_j)=\delta_{ij}, \\
&\phi_i\,\xi_j=-\phi_j\,\xi_i=\xi_k, \\
&\eta_i\circ\phi_j=-\eta_j\circ\phi_i=\eta_k,\\
&\phi_i\circ\phi_j-\eta_j\otimes\xi_i=-\phi_j\circ\phi_i+\eta_i\otimes\xi_j=\phi_k,
\end{align*}
for any cyclic permutation $(i, j, k)$ of $(1, 2, 3)$.

The  Reeb vector fields $\xi_s$ satisfy the following relations $g(\xi_i, \xi_j) = \delta_{ij}$ and 
$[\xi_i,\xi_j]=2\xi_k$. Thus, they span an integrable $3$-dimensional
distribution on a $3$-Sasakian manifold. Denote by ${\mathcal{F}}$ the 3-dimensional foliation generated 
by the  Reeb vector fields $(\xi_1, \xi_2, \xi_3)$. 

If $(S, g)$ is a compact $3$-Sasakian manifold, then the  Reeb vector fields $\xi_s$ are complete and the leaves of
the foliation ${\mathcal{F}}$ are compact. Hence, ${\mathcal{F}}$ is {\em quasi-regular}. The $3$-Sasakian structure on $S$
is said to be {\em regular} if  ${\mathcal{F}}$ is a regular foliation.

The following theorem was first proved by Ishihara \cite{Ishihara} in the regular case. The general version, that we recall here, 
was proved in \cite{BGMa} (see also \cite{BG-2}). First we recall that a $4n$-dimensional $(n>1)$ Riemannian manifold/orbifold is
quaternionic K\"ahler if it has holonomy group contained in $\mathrm{Sp}(n)\mathrm{Sp}(1)$, and a 4-dimensional quaternionic K\"ahler
manifold/orbifold is a self-dual Einstein  Riemannian manifold/orbifold. 

\begin{theorem} [\cite{BG-2}]\label{th:structure}
Let $(S, g)$ be a $3$-Sasakian manifold of dimension $4n + 3$
such that the {Reeb} vector fields $(\xi_1, \xi_2, \xi_3)$ are complete. 
Then the space of leaves $S /{\mathcal{F}}$ has the structure of
 a quaternionic K\"ahler orbifold
 $(\mathcal{O}, g_{\mathcal{O}})$ of dimension $4n$ such that the natural projection
 $\pi\colon S \,\rightarrow\,{\mathcal{O}}$ is a principal orbi-bundle with group $\mathrm{SU}(2)$ or $\mathrm{SO}(3)$,
  and  $\pi$ is a Riemannian orbifold submersion such that the scalar curvature of
$g_{\mathcal{O}}$ is $16n(n + 2)$.
\end{theorem}

\begin{proposition} [\cite{BG-2}]\label{prop: leaves}
Let $(S, g)$ be a $3$-Sasakian manifold such that the
{Reeb} vector fields $(\xi_1, \xi_2, \xi_3)$ are complete.
Denote by ${\mathcal{F}}$ the canonical three dimensional
foliation on $S$. Then,
\begin{enumerate}
\item[i)] The leaves of ${\mathcal{F}}$ are totally geodesic spherical space forms $\Gamma{\backslash}S^3$ of constant curvature
one, where $\Gamma\,\subset \mathrm{Sp}(1) =  \mathrm{SU}(2)$ is a finite subgroup.
\item[ii)] The $3$-Sasakian structure on $S$ restricts to a $3$-Sasakian structure on each leaf.
\item[iii)] The generic leaves are either $\mathrm{SU}(2)$ or $\mathrm{SO}(3)$.
\end{enumerate}
\end{proposition}

Let $(S, g)$ be a $3$-Sasakian manifold. Then, the isometry group $\mathrm{Iso}(S, g)$ of $(S, g)$ is non-trivial, and it has dimension $\geq 3$
since each Sasakian structure has an isometry group of dimension $\geq 1$.  Denote by $\mathrm{Aut}(S, g)\subset \mathrm{Iso}(S, g)$ 
the subgroup of the isometry group which preserves the $3$-Sasakian structure $(g,\xi_s, \eta_s, \phi_s; s=1,2,3)$ on $S$.
The group $\mathrm{Aut}(S, g)$ can be characterized as follows.

\begin{lemma}[\cite{BG-2}]\label{lem: aut3-sasaki}
Let $(S, g)$ be a $3$-Sasakian manifold, with the $3$-Sasakian structure $(g,\xi_s, \eta_s, \phi_s; s=1,2,3)$,
and let $f\in \mathrm{Iso}(S, g)$. Then, the following conditions
are equivalent
\begin{enumerate}
\item[i)] $f_{\star} \xi_s = \xi_s, \, s= 1, 2, 3$;
\item[ii)] $f^{\star} \eta_s = \eta_s, \, s= 1, 2, 3$; 
\item[iii)] $f_{\star} \circ \phi_s = \phi_s \circ f_{\star}, \, s= 1, 2, 3$; 
\item[iv)] $f\in \mathrm{Aut}(S, g)$.
\end{enumerate}
\end{lemma}

\begin{definition}\label{def:3-sasa-homogen}
A $3$-Sasakian manifold $(S, g)$ is said to be a {\em $3$-Sasakian homogeneous space} if the group $\mathrm{Aut}(S, g)$ 
acts transitively on $S$. 
\end{definition}

\begin{proposition}[\cite{BG-2}]\label{prop: homog-3-sasaki}
Let $(S, g)$ be a $3$-Sasakian homogeneous space of dimension $4n + 3$. Then,
all leaves are diffeomorphic and ${\mathcal{O}} = S /{\mathcal{F}}$ is a quaternionic K\"ahler manifold 
such that the natural projection  $\pi\colon S \,\rightarrow\,{\mathcal{O}}$ is a locally trivial Riemannian fibration.
Moreover, $\mathrm{Aut}(S, g)$ passes to the quotient and acts transitively on the space of leaves ${\mathcal{O}}$.
\end{proposition}

Note that a quaternionic K\"ahler manifold  is not necessarily K\"ahler,
as the name might suggest. Moreover, if ${\mathcal{O}}$ is a quaternionic K\"ahler manifold, 
then the scalar curvature $t$ of ${\mathcal{O}} $ is constant since it is Einstein.
Thus, there are three classes of examples of quaternionic K\"ahler manifolds corresponding to $t > 0$, $t=0$ and $t<0$. 

\begin{definition}
A  {\em positive quaternionic K\"ahler manifold} is a quaternionic K\"ahler manifold with
complete metric and with positive scalar curvature.
\end{definition}
We will use also the following properties of quaternionic K\"ahler manifolds.

\begin{theorem}[\cite{Salamon1}]\label{th:1-conn-positive}
Let $M$ be a compact quaternionic K\"ahler manifold of positive
scalar curvature. Then $\pi_{1}(M)=0$ and its odd Betti numbers are zero.
\end{theorem}

In order to show a classification of $3$-Sasakian homogeneous spaces, we recall that
any homogeneous positive quaternionic K\"ahler manifold is a symmetric space \cite{Alek}.
These homogeneous manifolds are referred to as {\em Wolf spaces} in recognition of \cite{Wolf}, 
and they are given by the $4n$-dimensional spaces of (real) dimension $4n$
$$
{\mathbb{H}}{\mathbb{P}}^{n}=\frac{{\mathrm{Sp}}(n+1)}{{\mathrm{Sp}}(n)\times 
{\mathrm{Sp}}(1)}, \quad {\mathbb{G}r}_2({\mathbb{C}}^{n+2}), \quad 
\widetilde{\mathbb{G}r}_4({\mathbb{R}}^{n+4})
$$
where $\widetilde{\mathbb{G}r}_4({\mathbb{R}}^{n+4})$ is the Grassmannian of oriented real 4-planes, and by the exceptional symmetric spaces
\begin{align*}
  & G I\,=\,\frac{{\mathrm{G}}_2}{{\mathrm{SO}}(4)}, \qquad F I\,=\,\frac{{\mathrm{F}}_4}{{\mathrm{Sp}}(3)\cdot{\mathrm{Sp}}(1)}, \qquad   
 E II\,=\,\frac{{\mathrm{E}}_6}{{\mathrm{SU}}(6)\cdot{\mathrm{Sp}}(1)}, \\
  &E VI\,=\, \frac{{\mathrm{E}}_7}{{\mathrm{Spin}}(12)\cdot{\mathrm{Sp}}(1)} \qquad  \text{and} \qquad E IX\,=\, \frac{{\mathrm{E}}_8}{{\mathrm{E}}_7\cdot{\mathrm{Sp}}(1)},
 \end{align*}
of real dimension $8$, $28$, $40$, $64$ and $112$, respectively.
This classification and Proposition \ref{prop: homog-3-sasaki} 
imply the following classification theorem of $3$-Sasakian homogeneous spaces.

\begin{theorem}[\cite{BG-2}]\label{class-homog-3-sasaki} 
Let $(S, g)$ be a $3$-Sasakian homogeneous space. Then $S$ is precisely one of the following 
homogeneous spaces:
   \begin{align*}
  &  \frac{{\mathrm{Sp}}(k+1)}{{\mathrm{Sp}}(k)} \cong S^{4k+3}, \quad 
  \frac{{\mathrm{Sp}}(k+1)}{{\mathrm{Sp}}(k)\times {\mathbb{Z}}_2} \cong {\mathbb{R}}{\mathbb{P}}^{4k+3}, \quad 
   \frac{{\mathrm{SU}}(n+2)}{{\mathrm{S}}\big({\mathrm{U}}(n)\times{\mathrm{U}}(1)\big)} , 
   \quad  \frac{{\mathrm{SO}}(m+4)}{{\mathrm{SO}}(m)\times {\mathrm{Sp}}(1)}, \\
   &   \frac{{\mathrm{G}}_2}{{\mathrm{Sp}}(1)} , \qquad \frac{{\mathrm{F}}_4}{{\mathrm{Sp}}(3)}, 
   \qquad \frac{{\mathrm{E}}_6}{{\mathrm{SU}}(6)}, \qquad  
   \frac{{\mathrm{E}}_7}{{\mathrm{Spin}}(12)}, \qquad \frac{{\mathrm{E}}_8}{{\mathrm{E}}_7},
   \end{align*}
where $k\geq 0$, $n \geq 1$ and $m \geq 3$ (the ones in the first line are called non-exceptional, and the ones in the second line
are called exceptional). For the first two cases when $k=0$, ${\mathrm{Sp}}(0)$ is the identity group.
Furthermore, the fiber $F$ over the Wolf space 
is ${\mathrm{Sp}}(1)$ only for $S^{4k+3}$. In all the other cases $F={\mathrm{SO}}(3)$.
\end{theorem}

Note that Theorem \ref{class-homog-3-sasaki} implies that any homogeneous $3$-Sasakian 
manifold is simply connected with the exception of the real projective space. 
Moreover, the dimension of the $3$-Sasakian homogeneous spaces corresponding to the exceptional Lie groups
are (in the order given in the previous Theorem) $11$, $31$, $43$, $67$ and $115$, respectively.

Let $(S,g)$ be a $3$-Sasakian regular manifold, and let 
$(M, g_{M})$ be the quaternionic K\"ahler manifold given as the space of
orbits. Then, Theorem \ref{th:structure} implies that there is a principal fiber bundle 
$F\,\rightarrow \,S\, \to\, M$ with $F=\SU(2)$ or $F=\SO(3)$.

Suppose first that $F=\SO(3)$. By formulas (3.6) and (3.7) of \cite{Konishi}, the principal fiber bundle $S \, \to \, M$ has a natural connection whose curvature can be described locally in terms of the quaternionic K\"ahler  structure  
on $M$. Indeed, locally on $M$ we have
three complex structures $J_1,J_2,J_3$  (with $J_3=J_{1}J_{2}=-J_{2}J_{1}$) each of them compatible with 
the Riemannian metric $g_{M}$, and with corresponding  (local) K\"ahler forms $\omega_1,\omega_2,\omega_3$. The
$\SO(3)$-bundle has a bundle of Lie algebras $\mathfrak{so}(3)$ with a frame $e_1,e_2,e_3$, and the curvature of
the natural connection is given by  $R=\omega_1\,e_1+\omega_2\, e_2+ \omega_3\, e_3$. The quaternionic K\"ahler
form is the 4-form $\Omega$ on $M$ given as 
 $$ 
\Omega=\tr (R\wedge R)= \omega_1\wedge \omega_1+\omega_2\wedge \omega_2+\omega_3\wedge \omega_3\,  .
 $$
This form is the representative of the first Pontryagin class 
 \begin{equation}\label{eqn:p1}
  p_1(S)=[\tr(R\wedge R)]=[\Omega]\in H^4(M,\ZZ). 
 \end{equation}

In the case that the principal fiber bundle $F \,\rightarrow \,S\,\to M$ has fiber $F=\SU(2)$, 
we can take the associated $\SO(3)$-bundle via the epimorphism $\SU(2)\to \SO(3) = \mathrm{SU}(2)/\mathbb{Z}_2$, equivalently
quotient by the center $\ZZ_2$. In this way we obtain another $3$-Sasakian manifold $S'=S/\ZZ_2$ with a fiber
bundle $F'=\SO(3)\,\rightarrow \,S' \,\to M$. 
The $\SU(2)$-bundle $S$ is characterised by the Euler class
$e(S)$ of the fibration, that is the second Chern class $c_2(S) \in H^4(S,\ZZ)$
and it is given by
  \begin{equation}\label{eqn:c2}
  e(S)=c_2(S) = - \frac{1}{4}[\Omega]. 
  \end{equation}
This can be proved as follows. Let $E\to M$ be the rank $2$ complex 
vector bundle associated to $S$, then $\End(E)\to M$ consisting of skew-hermitian
endomorphisms, is the rank $3$ real vector bundle associated to $S'$. Then $[\Omega]=p_1(S')=p_1(\End (E))= - c_2(\End(E)\ox \CC)=
-c_2(\End_\CC(E))=-c_2(E\ox E^*)$. To compute this, let $x_1,x_2$ be the Chern roots of $E$. So the Chern roots
of $E\ox E^*$ are $1, x_1-x_2$ and $x_2-x_1$.
So $[\Omega]=-c_2(E\ox E^*)=(x_1-x_2)^2= c_1(E)^2-4c_2(E)=-4c_2(E)=-4c_2(S)$ (see \cite{Milnor-Stasheff}).

Finally, for later use, we compute the Pontryagin class in the following situation. 
Suppose that we have a principal $F=\SO(3)$-bundle $S\to M$, and that there is a lifting to an $\U(2)$-bundle under the
epimorphism $\U(2)\to \SO(3)$, say 
$\widetilde{S} \to M$. The rank $2$ complex vector bundle $E\to M$ associated to 
$\widetilde{S}$ has Chern classes $c_1(E), c_2(E)$. The rank $3$ real vector  bundle associated to 
$S$ is $\End(E)\to M$, consisting of skew-hermitian
endomorphisms. Then,
 \begin{equation} \label{eqn:pont-chern}
p_1(S)=-c_2(E\ox E^*)=c_1(E)^2-4c_2(E)=c_1(\widetilde{S})^2-4c_2(\widetilde{S}). 
\end{equation}
For the Euler class $e(S)$ of the $\SO(3)$-bundle we have
\begin{equation} \label{eqn:eulerclassSO(3)}
e(S)=-\frac14 p_1(S).
\end{equation}

\section{Minimal models and formal manifolds}\label{min-models}

In this section we review some definitions and results about minimal models
and Massey products on smooth manifolds (see
\cite{DGMS, FHT, FM} for more details).

We work with {\em differential graded commutative algebras}, or DGAs,
over the field $\mathbb R$ of real numbers. The degree of an
element $a$ of a DGA is denoted by $|a|$. A DGA $(\mathcal{A},\,d)$ is said to be {\it minimal\/} if:
\begin{enumerate}
 \item $\mathcal{A}$ is free as an algebra, that is $\mathcal{A}$ is the free
 algebra $\bigwedge V$ over a graded vector space $V\,=\,\bigoplus_i V^i$, and

 \item there is a collection of generators $\{a_\tau\}_{\tau\in I}$
indexed by some well ordered set $I$, such that
 $|a_\mu|\,\leq\, |a_\tau|$ if $\mu \,< \,\tau$ and each $d
 a_\tau$ is expressed in terms of the previous $a_\mu$, $\mu\,<\,\tau$.
 This implies that $da_\tau$ does not have a linear part.
 \end{enumerate}

In our context, the main example of DGA is the de Rham complex $(\Omega^*(M),\,d)$
of a smooth manifold $M$, where $d$ is the exterior differential.

The cohomology of a differential graded commutative algebra $(\mathcal{A},\,d)$ is
denoted by $H^*(\mathcal{A})$. This space is
naturally a DGA with the product inherited from that on $\mathcal{A}$ while the differential on
$H^*(\mathcal{A})$ is identically zero. A DGA $(\mathcal{A},\,d)$ is  {connected} if 
$H^0(\mathcal{A})\,=\,\RR$, and it is {$1$-connected} if, in 
addition, $H^1(\mathcal{A})\,=\,0$.

Morphisms between DGAs are required to preserve the degree and to commute with the 
differential. We say that $(\bigwedge V,\,d)$ is a {\it minimal model} of a 
differential graded commutative algebra $(\mathcal{A},\,d)$ if $(\bigwedge V,\,d)$ 
is minimal and there exists a quasi-isomorphism, that is a morphism of differential graded algebras 
  $$
  \rho\,\colon\, {(\bigwedge V,\,d)}\,\longrightarrow\, {(\mathcal{A},\,d)}
 $$ 
inducing an isomorphism $\rho^*\,\colon\, H^*(\bigwedge 
V)\,\stackrel{\sim}{\longrightarrow}\, H^*(\mathcal{A})$ of cohomologies. 
A connected differential graded algebra has a minimal model unique up to isomorphism \cite{Halperin} (see
 \cite{DGMS,GM,Su} for the $1$-connected case).

A {\it minimal model\/} of a connected smooth manifold $M$
is a minimal model $(\bigwedge V,\,d)$ for the de Rham complex
$(\Omega^*(M),\,d)$ of differential forms on $M$. If $M$ is a simply
connected manifold, then the dual of the real homotopy vector
space $\pi_i(M)\otimes \RR$ is isomorphic to the space $V^i$ of generators in degree $i$, for any $i$.
The latter also happens when $i\,>\,1$ and $M$ is nilpotent, that
is, the fundamental group $\pi_1(M)$ is nilpotent and its action
on $\pi_j(M)$ is nilpotent for all $j\,>\,1$ (see~\cite{DGMS}).

We say that a DGA $(\mathcal{A},\,d)$ is a {\it model} of a manifold $M$
if $(\mathcal{A},\,d)$ and $M$ have the same minimal model. Thus, if $(\bigwedge V,\,d)$ is the minimal
model of $M$, we have 
$$
 (\mathcal{A},\,d)\, \stackrel{\nu}\longleftarrow\, {(\bigwedge V,\, d)}\, \stackrel{\rho}\longrightarrow\, (\Omega^{*}(M),\,d),
$$
where $\rho$ and $\nu$ are quasi-isomorphisms.

A minimal algebra $(\bigwedge V,\,d)$ is {\it formal} if there exists a
morphism of differential algebras $\psi\,\colon\, {(\bigwedge V,\,d)}\,\longrightarrow\,
(H^*(\bigwedge V),\,0)$ inducing the identity map on cohomology.
A DGA $(\mathcal{A},d)$ is formal if its minimal model is formal.
A smooth manifold $M$ is {formal} if its minimal model is
formal. Many examples of formal manifolds are known: spheres, projective
spaces, compact Lie groups, symmetric spaces, flag manifolds,
and compact K\"ahler manifolds. Recently, in \cite{Amann2} it is proved the following 

\begin{theorem}[\cite{Amann2}]\label{th:formalquat}
Compact positive quaternionic K\"ahler manifolds are formal.
\end{theorem}

\begin{remark}
Note that there are examples of non-formal homogeneous
spaces (see~\cite{Amann3} and references therein). Amann has proved  \cite{Amann3} that in 
every dimension $\geq 72$, there is an irreducible simply connected compact homogeneous space which is not formal. 
\end{remark}

The formality property of a minimal algebra is characterized as follows.

\begin{theorem}[\cite{DGMS}]\label{prop:criterio1}
A minimal algebra $(\bigwedge V,\,d)$ is formal if and only if the space $V$
can be decomposed into a direct sum $V\,=\, C\oplus N$ with $d(C) \,=\, 0$,
$d$ is injective on $N$ and such that every closed element in the ideal
$I(N)$ generated by $N$ in $\bigwedge V$ is exact.
\end{theorem}

This characterization of formality can be weakened using the concept of
$s$-formality introduced in \cite{FM}.

\begin{definition}\label{def:primera}
A minimal algebra $(\bigwedge V,\,d)$ is $s$-formal
($s > 0$) if for each $i\,\leq\, s$
the space $V^i$ of generators of degree $i$ decomposes as a direct
sum $V^i\,=\,C^i\oplus N^i$, where the spaces $C^i$ and $N^i$ satisfy
the following conditions:
\begin{enumerate}

\item $d(C^i) = 0$,

\item the differential map $d\,\colon\, N^i\,\longrightarrow\, \bigwedge V$ is
injective, and

\item any closed element in the ideal
$I_s=I(\bigoplus\limits_{i\leq s} N^i)$, generated by the space
$\bigoplus\limits_{i\leq s} N^i$ in the free algebra $\bigwedge
(\bigoplus\limits_{i\leq s} V^i)$, is exact in $\bigwedge V$.
\end{enumerate}
\end{definition}

A smooth manifold $M$ is $s$-formal if its minimal model
is $s$-formal. Clearly, if $M$ is formal then $M$ is $s$-formal for every $s\,>\,0$.
The main result of \cite{FM} shows that sometimes the weaker
condition of $s$-formality implies formality.

\begin{theorem}[\cite{FM}]\label{fm2:criterio2}
Let $M$ be a connected and orientable compact differentiable
manifold of dimension $2n$ or $(2n-1)$. Then $M$ is formal if and
only if it is $(n-1)$-formal.
\end{theorem}

One can check that any simply connected compact manifold 
is $2$-formal. Therefore, Theorem \ref{fm2:criterio2} implies that any
simply connected compact manifold of dimension at most six
is formal. (This result was proved earlier in \cite{N-Miller}.)

In order to detect non-formality, instead of computing the minimal
model, which is usually a lengthy process, one can use Massey
products, which are obstructions to formality. The simplest type
of Massey product is the triple Massey
product, which is defined as follows. Let $(\mathcal{A},\,d)$ be a DGA (in particular, it can be the de Rham complex
of differential forms on a smooth manifold). Suppose that there are
cohomology classes $[a_i]\,\in\, H^{p_i}(\mathcal{A})$, $p_i\,>\,0$,
$1\,\leq\, i\,\leq\, 3$, such that $a_1\cdot a_2$ and $a_2\cdot a_3$ are
exact. Write $a_1\cdot a_2=da_{1,2}$ and $a_2\cdot a_3=da_{2,3}$.
The {\it (triple) Massey product} of the classes $[a_i]$ is defined as
 $$
 \langle [a_1],[a_2],[a_3] \rangle \,=\, 
 [ a_1 \cdot a_{2,3}+(-1)^{p_{1}+1} a_{1,2}
  \cdot a_3] \in 
  \frac{H^{p_{1}+p_{2}+ p_{3} -1}(\mathcal{A})}{[a_1]\cdot H^{p_{2}+ p_{3} -1}(\mathcal{A})+ 
  [a_3]\cdot H^{p_{1}+ p_{2} -1}(\mathcal{A})}.
$$

Note that a Massey product $\langle [a_1],[a_2],[a_3] \rangle$ on $(\mathcal{A},\,d)$
is zero (or trivial) if and only if there exist $\widetilde{x}, \widetilde{y}\in \mathcal{A}$ such that
$a_1\cdot a_2=d\widetilde{x}$, \, $a_2\cdot a_3=d\widetilde{y}$\, and 
$[ a_1 \cdot \widetilde{y}+(-1)^{p_{1}+1}\widetilde{x}\cdot a_3]=0$.

We will use also the following property (see \cite{FIM} for a proof).
\begin{lemma} \label{lemm:massey-models}
Let $M$ be a connected differentiable manifold. Then, Massey products on $M$
can be calculated by using any model of $M$.
\end{lemma}

Moreover, we will use the following results.

\begin{lemma}[\cite{DGMS}] \label{lem:criterio1}
 If $M$ has a non-trivial Massey product, then $M$ is non-formal.
\end{lemma}

\begin{lemma}[\cite{FIM}] \label{lem:$3$-formal}
Let $M$ be a $7$-dimensional simply connected compact manifold with $b_2(M)\,\leq 1$.
Then, M is $3$-formal and so formal.
\end{lemma}

 \subsection*{Minimal models of  $\SU(2)$ and $\SO(3)$-fibrations}\label{subsect:fibrations}
Let $F\to E \to B$ be a fibration of simply connected spaces. Let 
$(\SA_B, \widetilde{d})$ be a model  (not necessarily minimal) of
the base $B$, and let $(\bigwedge V_F, d)$ be a minimal model of the fiber $F$. By \cite{RS}, a model
of $E$ is the \emph{KS-extension} $(\SA_B\otimes \bigwedge V_F, D)$, where $D$ is defined
as $Db=\widetilde{d}b$, for $b\in \SA_B$, and $Df= df+ \Theta(f)$, 
there $f\in V_F$, and $$\Theta:V_F\to \SA_B$$
is called the \emph{transgression map}. This is also true in the case that $F, E$ and $B$ are nilpotent
spaces and the fibration is nilpotent, that is $\pi_1(B)$ acts nilpotently in the homotopy
groups $\pi_j(F)$ of the fiber.

In the case that $E$ and $B$ are simply connected and $F=\SU(2)=S^3$ or $F=\SO(3)=\RP^3$, the fibration is nilpotent.
Note that $\pi_1(\SO(3))=\ZZ_2$ but it acts trivially on the higher homotopy groups, since
the antipodal map on $S^3$ is homotopic to the identity. The minimal model of $S^3$ is $(\bigwedge u,d)$,
with $|u|=3$ and $du=0$. Both spaces $S^3$ and $\RP^3$ are rationally homotopy equivalent, so a fibration
$\SO(3)\,\rightarrow \,E\,\to B$ is a rational $S^3$-fibration (that is, after rationalization of the spaces, it becomes
a fibration). The transgression map $\Theta$ is such that $\Theta(u) \in \SA_B^4$ is a closed 
element of degree $4$ defining the Euler class  $e(E)$ of the fibration. 

\section{$\SU(2)$ and $\SO(3)$-bundles over the  complex Grassmannian ${\mathbb{G}r}_2({\mathbb{C}}^{n+2})$} \label{sect:grassmann1}

Now it is our purpose to prove the formality of all the $3$-Sasakian homogeneous spaces.
By Theorem \ref{class-homog-3-sasaki} we know that, except for the sphere $S^{4n+3}$, such an space is the total space of an 
$\SO(3)=\RP^3$-bundle over a Wolf space. The $\SO(3)$-bundles over the exceptional Wolf spaces will be 
treated in Section \ref{sect:exceptional}. For the $\SO(3)$-bundles over the non-exceptional Wolf spaces, we note that 
the sphere $S^{4n+3}$ and the projective space ${\mathbb{R}}{\mathbb{P}}^{4n+3}$ are formal, 
so it is sufficient to prove the formality of the spaces
${\mathrm{SU}}(n+2)\,/{\mathrm{S}}\big({\mathrm{U}}(n)\times{\mathrm{U}}(1)\big)$ and
${\mathrm{SO}}(n+4)\,/{(\mathrm{SO}}(n)\times {\mathrm{Sp}}(1))$. In this section we deal with the
first case and, more generally, we study the formality of the total space of $\SU(2)$ and $\SO(3)$-bundles over the complex Grassmannian
 $$
  {\mathbb{G}r}_2({\mathbb{C}}^{n+2})=\frac{{\mathrm{SU}}(n+2)}{{\mathrm{S}}\big({\mathrm{U}}(n)\times{\mathrm{U}}(2)\big)}.
 $$
Its cohomology ring is given by (see~\cite{Amann1})
  \begin{equation}\label{eqn:H-Gr}
	H^*\big({\mathbb{G}r}_2({\mathbb{C}}^{n+2})\big) =
	H^\ast(BT)^{W(\U(n)\times \U(2))}/H^{>0}(BT)^{W(\U(n+2))}\, ,
  \end{equation}
where $BT$ is the classifying space of a maximal torus of $\U(n+2)$, and
$W(G)$ denotes the Weyl group of a Lie group $G$. Thus the classes of
the cohomology of ${\mathbb{G}r}_2({\mathbb{C}}^{n+2})$ are symmetric
polynomials of $H^\ast(BT)=\mathbb{Q}[x_1,\ldots,x_n,x_{n+1},x_{n+2}]$, where each 
generator $x_i$ has degree 2, for $1 \leq i \leq (n+2)$.
Denote by $y_1,y_2$ the classes $x_{n+1},x_{n+2}$, respectively. On one
hand, $H^\ast(BT)^{W(\U(n)\times \U(2))}$ is generated by the symmetric
polynomials $\tau_n$ of $x_1,\ldots,x_n$, and by the symmetric
polynomials $\sigma_n$ of $y_1,y_2$. On the other hand,
$H^{>0}(BT)^{W(\U(n+2))}$ is generated by the symmetric polynomials of
$x_1,\ldots,x_n,y_1,y_2$. This gives the relation $\sigma\,
\tau=1$, where $\sigma=1+\sigma_1+\sigma_2+\ldots$ and
$\tau=1+\tau_1+\tau_2+\ldots$  Denote $l=y_1+y_2$ and $x=y_1\,y_2$. 
Here and in what follows, $y_1y_2$ stands for the cup product $y_1 \cup y_2$, and so on.
Thus,
\begin{equation}\label{cohom:complexgrass}
H^*\big({\mathbb{G}r}_2({\mathbb{C}}^{n+2})\big) =\mathbb{Q}[l,x]/(\sigma_{n+1},\sigma_{n+2}), 
\end{equation}
where $|l|=2$, $|x|=4$, and $\sigma_r$ $(r\geq 0)$ is the cohomology class of degree $2r$ that is defined recursively by 
 \begin{equation}\label{eqn:recursive}
\left\{
\begin{array}{l}
\sigma_0=1,\\
\sigma_1=-l,\\
 \sigma_r=-l\, \sigma_{r-1}-x\,\sigma_{r-2}, \qquad r\geq 2.
\end{array}
\right.
 \end{equation}

In the following Lemma, we determine the expression of $\sigma_r$ in terms of the cohomology classes $l$, $x$ and their cup products.

\begin{lemma}\label{lem:gen_expr}
The cohomology class $\sigma_r$ $(r\geq 0)$ on  ${\mathbb{G}r}_2({\mathbb{C}}^{n+2})$ has the following expression
\begin{equation}\label{eq:sigma}
\sigma_r=\sum_{k=0}^{\lfloor r/2\rfloor}(-1)^{r+k}\binom{r-k}{k} l^{r-2k}\,x^k.
\end{equation}
\end{lemma}

\begin{proof}
We proceed by induction on $r\geq 0$. It is clear that \eqref{eq:sigma} is true for $r=0,1$. 
Assume that \eqref{eq:sigma} holds for $s\leq r$. Then,
using the recursive definition of $\sigma_k$ given by \eqref{eqn:recursive} and the induction hypothesis we have
\begin{align*}
\sigma_{r+1}&=-l\sigma_r-x\sigma_{r-1}\\
&=-l\sum_{k=0}^{\lfloor r/2\rfloor}(-1)^{r+k}\binom{r-k}{k}l^{r-2k}x^k\, -\, x \sum_{k=0}^{\lfloor (r-1)/2\rfloor}(-1)^{r-1+k}\binom{r-1-k}{k}l^{r-1-2k}x^k\\
&=-\sum_{k=0}^{\lfloor r/2\rfloor}(-1)^{r+k}\binom{r-k}{k}l^{r-2k+1}x^k
-\sum_{k=0}^{\lfloor (r-1)/2\rfloor}(-1)^{r-1+k}\binom{r-1-k}{k}l^{r-1-2k}x^{k+1}\\
&=-\sum_{k=0}^{\lfloor r/2\rfloor}(-1)^{r+k}\binom{r-k}{k}l^{r-2k+1}x^k
-\sum_{k=1}^{\lfloor (r-1)/2\rfloor+1}(-1)^{r+k}\binom{r-k}{k-1}l^{r+1-2k}x^{k}.
\end{align*}
Thus,
 \begin{equation} \label{eqn:induct-recursive}
  \begin{aligned} 
\sigma_{r+1}&=-\sum_{k=1}^{\lfloor r/2\rfloor}(-1)^{r+k}\left(\binom{r-k}{k-1}+\binom{r-k}{k}\right)l^{r-2k+1}x^k-(-1)^r\binom{r}0l^{r+1}\\
&\hphantom{=\;}+\epsilon (-1)^{r+\lfloor (r-1)/2\rfloor}\binom{r-\lfloor(r-1)/2\rfloor-1}{\lfloor(r-1)/2\rfloor}x^{\lfloor(r-1)/2\rfloor+1},
 \end{aligned} 
\end{equation}
where $\epsilon\in \{0, 1\}$, and $\epsilon =1$ if and only if $r$ is odd, otherwise $\epsilon=0$. Clearly
if $r$ is odd,  $\lfloor (r-1)/2\rfloor=(r-1)/2$ and $\binom{r-\lfloor(r-1)/2\rfloor-1}{\lfloor(r-1)/2\rfloor}=1$.
Moreover, it is well known that $\binom{r-k}{k-1}+\binom{r-k}{k}=\binom{r+1-k}{k}$ and $\binom{r}0=\binom{r+1}{0}$. 
Substituting these equalities into \eqref{eqn:induct-recursive}, we obtain
 $$
\sigma_{r+1} = \sum_{k=0}^{\lfloor (r+1)/2\rfloor} (-1)^{r+1+k}\binom{r+1-k}{k} l^{r+1-2k}x^k.
 $$
\end{proof}

For the quaternionic K\"ahler form on ${\mathbb{G}r}_2({\mathbb{C}}^{n+2})$ we have:

\begin{proposition}\label{prop:formgrassc}
Let $\Omega$ be the quaternionic K\"ahler form on ${\mathbb{G}r}_2({\mathbb{C}}^{n+2})$. Then, in terms of the 
generators $l$ and $x$ of $H^*\big({\mathbb{G}r}_2({\mathbb{C}}^{n+2})\big)$ given by \eqref{cohom:complexgrass},
the de Rham cohomology class $[\Omega]\in H^4({\mathbb{G}r}_2({\mathbb{C}}^{n+2}))$ of $\Omega$ is 
 $$
 [\Omega]\,= \, l^2-4x . 
 $$
\end{proposition}

\begin{proof}
The cohomology class defined by $\Omega$ in $H^4({\mathbb{G}r}_2({\mathbb{C}}^{n+2}))$ determines
the Pontryagin class of the principal $\SO(3)$-bundle
$$
 \mathrm{SO}(3) = \mathrm{SU}(2)/\mathbb{Z}_2  \,\hookrightarrow \,
S=\frac{{\mathrm{SU}}(n+2)}{{\mathrm{S}}\big({\mathrm{U}}(n)\times{\mathrm{U}}(1)\big)} \,\longrightarrow\,
 {\mathbb{G}r}_2({\mathbb{C}}^{n+2})
 =\frac{{\mathrm{SU}}(n+2)}{{\mathrm{S}}\big({\mathrm{U}}(n)\times{\mathrm{U}}(2)\big)},
 $$
which gives the natural $3$-Sasakian homogeneous space $S$, where $\U(1)\subset \U(2)$ diagonally. 
We can lift this $\SO(3)$-bundle to an $\U(2)$-bundle by considering the complex Grassmanian ${\mathbb{G}r}_2({\mathbb{C}}^{n+2})$ as the
quotient space
${\mathbb{G}r}_2({\mathbb{C}}^{n+2})={\mathrm{U}}(n+2)\,/ \big({\mathrm{U}}(n)\times{\mathrm{U}}(2)\big)$,
and taking the principal $\U(2)$-bundle
 $$
  \mathrm{U}(2)   \,\hookrightarrow \, \widetilde S={\mathrm{U}}(n+2)/ \mathrm{U}(n)  \,\longrightarrow\,
 {\mathbb{G}r}_2({\mathbb{C}}^{n+2})={\mathrm{U}}(n+2)\,/ \big({\mathrm{U}}(n)\times{\mathrm{U}}(2)\big).
 $$
By the description in (\ref{eqn:H-Gr}), $y_1$ and $y_2$ are the Chern roots of $\U(2)<\U(n+2)$ embedded as the last two
entries. Then, $c_1(\widetilde S)=y_1+y_2=l$ and $c_2(\widetilde S)= y_1y_2=x$. 
Thus, by \eqref{eqn:p1} and \eqref{eqn:pont-chern}, 
 $$
 [\Omega]= p_1(S)= c_1(\widetilde S)^2-4c_2(\widetilde S)= l^2-4x.
 $$
\end{proof}

\begin{theorem} \label{th: 3-sasa-complexgr-1}
Consider  the principal  $F$-fiber bundle 
$F \,\rightarrow \,S \,\rightarrow\, {\mathbb{G}r}_2({\mathbb{C}}^{n+2})$ with 
$F=\SU(2)$ or $F=\SO(3)$, and Euler class $a l^2 + b x$, where $a, b\in\mathbb{Q}$. If $b\neq 0$, then $S$ is formal. In particular, the $3$-Sasakian homogeneous space 
${\mathrm{SU}}(n+2)/{\mathrm{S}}\big({\mathrm{U}}(n)\times{\mathrm{U}}(1)\big)$ is formal. 
\end{theorem}

\begin{proof}
We can assume that $n\,\geq 2$. Indeed, if $n\,=\,0$, then $S$ is formal by Theorem \ref{fm2:criterio2} and, 
if $n\,=\,1$, then $S$ is formal by Lemma~\ref{lem:$3$-formal}.

 Since ${\mathbb{G}r}_2({\mathbb{C}}^{n+2})$ is a compact positive quaternionic K\"ahler manifold, ${\mathbb{G}r}_2({\mathbb{C}}^{n+2})$
is simply connected \cite{Salamon1}. Then, according with section \ref{min-models}, 
the fibre bundle $F \rightarrow \, S \,\rightarrow\,{\mathbb{G}r}_2({\mathbb{C}}^{n+2})$, with $F=\SU(2)$ or $F=\SO(3)$ and Euler class
$a l^2 + b x$,  is a rational $S^3$-fibration. Thus~\cite{RS},  if $(\mathcal{A}, d_{\mathcal{A}})$ is a model 
of ${\mathbb{G}r}_2({\mathbb{C}}^{n+2})$, we have that
$(\mathcal{A} \otimes \bigwedge(u), d)$, with $|u|\,=\,3$, $d\vert_{\mathcal{A}}\,=\,d_{\mathcal{A}}$
and $du\,=\,a l^2 + b x$, is a model of $S$.

We know that ${\mathbb{G}r}_2({\mathbb{C}}^{n+2})$ is formal since it is a 
symmetric space of compact type (see also Theorem \ref{th:formalquat}).
Thus, a model of ${\mathbb{G}r}_2({\mathbb{C}}^{n+2})$ is
$\big(H^*({\mathbb{G}r}_2({\mathbb{C}}^{n+2})), 0\big)$, where $H^*({\mathbb{G}r}_2({\mathbb{C}}^{n+2}))$ is the
cohomology algebra of ${\mathbb{G}r}_2({\mathbb{C}}^{n+2})$  defined by \eqref{cohom:complexgrass}. Hence, 
a model of $S$ is the differential algebra
$(H^*({\mathbb{G}r}_2({\mathbb{C}}^{n+2}))\otimes \bigwedge(u), d)$, where $u$ has degree $3$ and 
$du=a l^2 + b x$.

Since $b\not=0$, $x = -(a\,/b)\,l^2$ on $H^*(S)$. Using that $\sigma_{n+1} = 0=\sigma_{n+2}$ on ${\mathbb{G}r}_2({\mathbb{C}}^{n+2})$,
from Lemma \ref{lem:gen_expr} we have 
 \begin{align*}
 & l^{n+1} \sum_{k=0}^{\lfloor (n+1)/2\rfloor}(-1)^{n+1+k}\binom{n+1-k}{k} \left(-\frac{a}{b}\right)^k = 0,  \\
 & l^{n+2} \sum_{k=0}^{\lfloor (n+2)/2\rfloor}(-1)^{n+k}\binom{n+2-k}{k} \left(-\frac{a}{b}\right)^k =0 .
 \end{align*}
One of the coefficients should be non-zero. If $a=0$, it is clear (because only the summand with 
$k=0$ contributes). 
If $a\neq 0$, we can run the recursive relations
(\ref{eqn:recursive}) backwards with $x=(-a/b)l^2$. If both coefficients were zero, then all $\sigma_r|_{x=(-a/b)l^2}=0$,
but this contradicts that $\sigma_0=1$. Therefore we have that either 
 $l^{n+1} = 0$ or $l^{n+2} = 0$ on $H^*(S)$. 

We deal first with the possibility that 
$l^{n+1} = 0$. Since $x = -\frac{a}{b} l^2$ on $S$, the cohomology of $S$ up to the degree $2n+1$ is
  $$
H^{0}(S)=\langle 1\rangle\,, \qquad H^{2i+1}(S) =0\,, \,\, 0 \leq i \leq n, \qquad H^{2j}(S)=\langle l^j\rangle , \,\, 1 \leq j \leq n\,.
  $$
By Poincar\'e duality,  $H^{4n+3-2j}(S)=\langle PD(l^j)\rangle$, for $1 \leq j \leq n$, where $PD(l^j)$ denotes the Poincar\'e dual of 
$l^j$.  Therefore, the minimal model of $S$ must be a differential graded algebra
$(\bigwedge V,d)$, being $\bigwedge V$ the free algebra of the form
$\bigwedge V=\bigwedge(a_2, v_{2n+1})\otimes \bigwedge V^{\geq (2n+2)}$, where $|a_2|=2$,
$|v_{2n+1}|=2n+1$, and $d$ is defined by $da_2=0$, $dv_{2n+1}=a_{2}^{n+1}$. 
According with Definition~\ref{def:primera}, we get 
$N^j=0$ for $j\leq 2n$, thus $S$ is $2n$-formal. Moreover, $S$ is 
$(2n+1)$-formal. In  fact, take $\alpha\in I(N^{\leq {2n+1}})$ a closed
element in $\bigwedge V$. As $H^*(\bigwedge V)=H^*(S)$ has only non-zero cohomology in even degrees $2j$ with $1 \leq j \leq n$, and 
in odd degrees $4n+3-2j$ with $1 \leq j \leq n$, 
it must be $|\alpha|= 4n+3-2j$, where $1 \leq j \leq n$. Hence $\alpha=a_{2}^{n-j+1}\,v_{2n+1}$, which is not closed.
So, according with Definition \ref{def:primera}, $S$ is $(2n+1)$-formal and, by Theorem
\ref{fm2:criterio2}, $S$ is formal.

Suppose now that $l^{n+1}\not=0$ but $l^{n+2} = 0$.
Note that $l^{n+1}\neq 0$ if and only if $\sigma_{n+1}=\tau\,(al^2+bx),$ for some non-zero cohomology 
class $\tau \in H^{2n-2}(\mathbb{G}r_2(\CC^{n+2}))$.  Then $\tau\, u$ is closed on $S$ because $d(\tau\,u)=\sigma_{n+1}$, 
and hence $H^{2n+1}(S)=\langle \tau\,u\rangle$. Clearly, $\tau\,u$ is not exact, since 
the image of the differential map $d$ is contained in $\mathcal{A}$.
Then, the minimal model of $S$ must be a differential graded algebra
$(\bigwedge V,d)$, being $\bigwedge V=\bigwedge(a_2, a_{2n+1}, v_{2n+3})\otimes \bigwedge V^{\geq (2n+4)}$, 
where $|a_2|=2$, $|a_{2n+1}|=2n+1$, 
$|v_{2n+3}|=2n+3$, and the differential $d$ is given by $da_2=0=da_{2n+1}$ and $dv_{2n+3}=a_{2}^{2n+2}$. 
According with Definition~\ref{def:primera}, we get 
$N^j=0$ for $j\leq 2n+1$, thus the manifold $S$ is $(2n+1)$-formal and, by Theorem~\ref{fm2:criterio2},  $S$ is formal.

From \eqref{eqn:eulerclassSO(3)}, Theorem \ref{class-homog-3-sasaki} and Proposition \ref{prop:formgrassc}, 
the $3$-Sasakian homogeneous space 
$S={\mathrm{SU}}(n+2)\,/{\mathrm{S}}\big({\mathrm{U}}(n)\times{\mathrm{U}}(1)\big)$ is the $\SO(3)$-bundle
$\SO(3)\rightarrow \, S \,\rightarrow\, {\mathbb{G}r}_2({\mathbb{C}}^{n+2})$ with Euler
class $-\frac14(l^2-4x)$, and so $S$ is formal. Even more, using \eqref{eq:sigma} for $\sigma_{n+1}$,
one can check that $l^{n+1}=0$ on $S$. So a minimal model of $S$ is the minimal model previously described for  $l^{n+1}=0$.
\end{proof}

\begin{theorem} \label{th: nonformal-complexgr-1}
A principal $\SU(2)$ or $\SO(3)$-fiber bundle 
$F \rightarrow \, S \,\rightarrow\, {\mathbb{G}r}_2({\mathbb{C}}^{n+2})$ with Euler class $a l^2$,  where $a\in\mathbb{Q}$, is formal if and only if $n$ is odd or $a=0$. 
\end{theorem}

\begin{proof}
If $a=0$ then $S$ is rationally equivalent to $S^3\times {\mathbb{G}r}_2({\mathbb{C}}^{n+2})$, which
is formal being the product of two formal manifolds.

Suppose now that $a\neq 0$. Proceeding as in the proof of Theorem~\ref{th: 3-sasa-complexgr-1}, a model of
$S$ is  $(H^\ast(\mathbb{G}r_2(\CC^{n+2}))\otimes \bigwedge(u),d)$, where $|u|=3$ and $du=al^2$. The 
cohomology of $S$ up to degree $2n+1$ is  $H^{0}(S)=\langle 1\rangle\,,$
$H^{2k}(S)=\langle l^\epsilon x^{\lfloor k/2\rfloor}\rangle$, $H^{2k+1}(S)=0$,
where $\epsilon=k-2\lfloor k/2\rfloor$ and $k=0,1,\ldots,n$. 

Suppose that $n$ is odd.  Since the cohomology class $\sigma_{n+1}$ is zero in
$\mathbb{G}r_2(\CC^{n+2})$, the explicit expression of $\sigma_{n+1}$
in Lemma~\ref{lem:gen_expr} implies that $x^{(n+1)/2}=0$ in $S$. Thus,
the minimal model of $S$ must be a differential graded algebra
$(\bigwedge V,d)$ where 
$\bigwedge V=\bigwedge(a_2, a_4)\otimes\bigwedge(v_3,v_{2n+1})\otimes \bigwedge V^{\geq 2n+2}$, where $|a_i|=i$ with $i=2,4$, $|v_j|=j$ with $j=3, 2n+1$, and the differential map is defined by
$da_i=0$, $dv_3=a_{2}^2$ and $dv_{2n+1}=a_{4}^{(n+1)/2}$.
Now take $\alpha$ a closed element in the ideal generated by $I(N^{\leq 3})$ in $\bigwedge V$.
 Then, $\alpha$ is of the form $\alpha=a_{2}^{p} \, v_3$ which is not closed, for any integer 
number $p\geq 1$.
Therefore $S$ is $3$-formal, and so $S$ is $2n$-formal because the spaces $N^j$ are zero for $4\leq j \leq 2n$.
Let us prove that $S$ is
$(2n+1)$-formal. 
Let $\alpha$ be an element of the ideal $I(N^{\leq 2n+1})$, and let us suppose that $\alpha$ is closed and homogeneous.
Then, $|\alpha| > (2n+1)$ must be odd by the cohomology of $S$. Thus, $\alpha$ is of the form 
$\alpha=P_1 \,v_3+P_2 \,v_{2n+1}$, where $P_1,P_2 \in \bigwedge(a_2, a_4)$.
The equality
$d\alpha=0$ implies that there exists $P\in \bigwedge(a_2, a_4)$ such that $P_1=P\, a_{4}^{(n+1)/2}$ 
and $P_2=-P\,a_{2}^2$. Hence $\alpha=d(P\,v_3 \,v_{2n+1})$ is exact. Therefore $S$ is
$(2n+1)$-formal, and by Theorem~\ref{fm2:criterio2} it is formal. 

If $n$ is even, the explicit expression of $\sigma_{n+1}$
in Lemma~\ref{lem:gen_expr} implies that $l\, x^{n/2}=0$ on $S$ since $\sigma_{n+1}=0$ . Then,
$\langle l,l,x^{n/2}\rangle$ defines a triple Massey product. By
Lemma~\ref{lemm:massey-models}, we compute this Massey product in our model.
Here, $l^2=du$ and 
 $$
  l \,x^{n/2}=d\left(\sum_{k=0}^{\lfloor (n-1)/2\rfloor} (-1)^{n+k}
  \binom{n+1-k}{k} l^{n-1-2k}x^ku\right),
  $$
by Lemma~\ref{lem:gen_expr}. So the triple 
Massey product $\langle l,l,x^{n/2}\rangle=\xi\, u$, where
 $$
  \xi=x^{n/2}-\sum_{k=0}^{\lfloor (n-1)/2\rfloor} (-1)^{n+k}
  \binom{n+1-k}{k} l^{n-1-2k}x^k.
  $$

Let us see that $\xi\, u$ is not exact. For
$\xi\, u$ to be exact in the model $\mathcal{A}\otimes
\bigwedge(u)$ we need that  $\xi=0$ in $\mathcal{A}=H^\ast(\mathbb{G}r_2(\CC^{n+2}))$, due to
the image of the differential map is contained in $\mathcal{A}$. Since $|\xi|=2n$, it cannot be a combination of $\sigma_{n+1}$ and
$\sigma_{n+2}$, of degrees~$2n+2$ and~$2n+4$, respectively. Hence, $\xi$
does not belong to the ideal $(\sigma_{n+1},\sigma_{n+2})$ and thus it
is non-zero. Therefore, $S$ is non-formal.
\end{proof}

\begin{remark}
These results can be extended to the family ${S}(\mathbf{p})$
of $3$-Sasakian manifolds for $\mathbf{p}=(p_1,p_2,\ldots,p_{n+2})\in
\ZZ^{n+2}_{>0}$ introduced in~\cite{BGMa}. They are defined as 
$\U(n+2)/(\U(n)\times \U(1))$, where the action of $\U(n)\times \U(1)$ on $\U(n+2)$ is defined by
  $$
  ((B,\lambda),A) \mapsto \begin{pmatrix}\lambda^{p_1} \\ &
 \ddots \\ & & \lambda^{p_{n+2}}\end{pmatrix}\cdot A\cdot \begin{pmatrix}
 B & 0 \\ 0 & I_2 \end{pmatrix}.
 $$
In particular,  this family includes the homogeneous space of Theorem~\ref{th:
3-sasa-complexgr-1} by letting $\mathbf{p}=(1,1,\ldots,1)\in
\ZZ^{n+2}$. If $\mathbf{p}\neq (1,\ldots,1)$, the manifold
${S}(\mathbf{p})$ is inhomogeneous. 

A description of their cohomology with integer coefficients is in~\cite[Theorem E]{BGMa}:
  $$
  H^\ast({S}(\mathbf{p}),\mathbb{Z})\cong 
  \left(  \frac{\mathbb{Z}[b_2]}{\la b^{n+2}_2 \ra }\otimes E[f_{2n+3}]\right)/
  \la \sigma_{n+1}(\mathbf{p})b_2^{n+1} , f_{2n+3}\, b_2^{n+1}\ra ,
  $$
where $|b_2|=2$ and $|f_{2n+3}|=2n+3$, and
$\sigma_{n+1}(\mathbf{p})$ denotes the $(n+1)$-th elementary symmetric polynomial on
$\mathbf{p}$. The first relation implies that ${S}(\mathbf{p})$
are not homotopic equivalent \cite[Corollary 8.1]{BGMa}. Nevertheless, if we take rational
coefficients, we have $b_2^{n+1}=0$ in $H^\ast({S}(\mathbf{p}))$ for any
$\mathbf{p}$.
Hence, the cohomology ring of ${S}(\mathbf{p})$ with rational
coefficients depends only on the length of $\mathbf{p}$. In particular,
Theorem~\ref{prop:formgrassc} implies that they are formal.
\end{remark}

\section{$\SU(2)$ and $\SO(3)$-bundles over the oriented real Grassmannian $\widetilde{\mathbb{G}r}_4({\mathbb{R}}^{n+4})$} \label{sect:grassmann2}

 In this section we show that the $3$-Sasakian homogeneous space 
${\mathrm{SO}}(n+4)\,/({\mathrm{SO}}(n)\times {\mathrm{Sp}}(1))$, where $n \geq 3$,
is formal. More generally, we study the formality of the total space of $\SU(2)$ and $\SO(3)$-bundles over the 
oriented real Grassmannian manifold
 $$
 \widetilde{\mathbb{G}r}_4({\mathbb{R}}^{n+4})=\frac{{\mathrm{SO}}(n+4)}{{\mathrm{SO}}(n)\times {\mathrm{SO}}(4)}.
 $$

To make explicit the cohomology ring of $\widetilde{\mathbb{G}r}_4({\mathbb{R}}^{n+4})$, we distinguish the case 
when $n$ is even and the case when $n$ is odd.

\subsection{$n$ is even.}  Take $n=2m \geq 4$.
Then, the cohomology of $\widetilde{\mathbb{G}r}_4({\mathbb{R}}^{n+4})$ is given by (see~\cite{Amann1})
 $$
 H^*\big(\widetilde{\mathbb{G}r}_4({\mathbb{R}}^{n+4})\big) =
 H^\ast(BT)^{W(\SO(n)\times \SO(4))}/H^{>0}(BT)^{W(\SO(n+4))}
 $$
where $BT$ is the classifying space of a maximal torus $T$ in $\SO(n+4)$, and
$W(G)$ denotes the Weyl group of the Lie group $G$. So the 
cohomology classes on $\widetilde{\mathbb{G}r}_4({\mathbb{R}}^{n+4})$ can be viewed as
symmetric polynomials of elements in 
 $H^\ast(BT)=\mathbb{Q}[x_1,x_2,\ldots,x_{m}, x_{m+1}, x_{m+2}]$, where
$x_i$ has degree $2$, for $1 \leq i\leq (m+2)$. If we denote 
by $y_1$ and $y_2$ the classes $x_{m+1}$ and $x_{m+2}$, respectively, the symmetric
polynomials $\tilde\tau_k$ of $x_1^2,\ldots,x_{m}^2$ and the symmetric
polynomials $\tilde\sigma_k$ of $y_1^2,y_2^2$, for $k\geq0$, are the generators of
$H^\ast(BT)^{W(\SO(n)\times \SO(4))}$. As 
$n$ and~$4$ are even, 
$x_1\cdots x_{m}$ and $y_1y_2$ are also invariant by the action of the group
$W(\SO(n)\times \SO(4))$.
On the other hand, the symmetric polynomials $\sigma_k$ of
$x_1^2,\ldots,x_{m}^2, y_1^2, y_2^2$ and $x_1\,\cdots\, x_{m} \,y_1\, y_2$ are
invariant by $W(\SO(n+4))$.
These are the generators of $H^{>0}(BT)^{W(\SO(n+4))}$ 
with the relations $\left(x_1\cdots x_{m}\right)\,\left(y_1y_2\right)=0$ and
$ \widetilde\tau\,\widetilde\sigma=1$, where
$ \widetilde\tau=1+ \widetilde\tau_1+ \widetilde\tau_2+\ldots$ and
$ \widetilde\sigma=1+ \widetilde\sigma_1+ \widetilde\sigma_2+\ldots$. Hence, we can take 
$l=y_1^2+y_2^2$, $x=y_1\,y_2$ and $z=x_1\,  x_2\, \cdots\, x_{m}$ as the generators of the cohomology of
$\widetilde{\mathbb{G}r}_4({\mathbb{R}}^{n+4})$, and so
     \begin{equation}\label{cohom:realgrasseven}
    H^*\big(\widetilde{\mathbb{G}r}_4({\mathbb{R}}^{n+4})\big) =\mathbb{Q}[l,x,
    z]/(x z, z^2 - \widetilde\sigma_{m},  \widetilde\sigma_{m+1}), 
    \end{equation}
where $|l| = |x|=4$, 
$|z|=2m$ and $ \widetilde\sigma_r$ $(r\geq 0)$ is the  cohomology
class of degree $4r$ that is defined recursively by 
  \begin{equation}\label{eq:sigmaireal}
  \left\{ \begin{array}{l}
  \widetilde\sigma_0=1,\\
   \widetilde\sigma_1=-l,\\
   \widetilde\sigma_r=-l \widetilde\sigma_{r-1}-x^2 \widetilde\sigma_{r-2}, \qquad r\geq 2\,.
  \end{array} \right.
  \end{equation}
 A similar proof to the one made for Lemma~\ref{lem:gen_expr} shows
\begin{equation}\label{eq:realsigma}
  \widetilde\sigma_r=\sum_{k=0}^{\lfloor r/2\rfloor}(-1)^{r+k}\binom{r-k}{k} l^{r-2k} x^{2k}, \qquad  r \geq 0\,.
 \end{equation}

\begin{proposition}\label{prop:formgrassr}
Let $\Omega$ be  the quaternionic K\"ahler form on $\widetilde{\mathbb{G}r}_4({\mathbb{R}}^{n+4})$ with $n=2m\geq 4$.
Then, in terms of the generators $l, x, z$ of $H^*(\widetilde{\mathbb{G}r}_4({\mathbb{R}}^{n+4}))$ given by (\ref{cohom:realgrasseven}),
the de Rham cohomology class $[\Omega]\in H^4(\widetilde{\mathbb{G}r}_4({\mathbb{R}}^{n+4}))$ is $[\Omega]=l+2x$.
\end{proposition}

\begin{proof}
By \eqref{eqn:p1}, the cohomology class $[\Omega]$ equals the Pontryagin class of the $\SO(3)$-fiber bundle 
which gives the $3$-Sasakian homogeneous
manifold $S={\mathrm{SO}}(n+4)/({\mathrm{SO}}(n)\times {\mathrm{Sp}}(1))$. To make the description more explicit,
recall that the double cover of $\mathrm{SO}(4)$ is $\mathrm{Spin}(4)\cong \mathrm{SU}(2)_+\times \mathrm{SU}(2)_-$, where
we label the two copies of $\mathrm{SU}(2)$ according to whether they act on $\bigwedge^2_\pm \RR^4$ of the
standard representation $\RR^4$ of $\SO(4)$. 
By projecting onto $\mathrm{SU}(2)_+$, we have a map $\mathrm{Spin}(4) \to \mathrm{SU}(2)_+$, which descends to 
a map $\mathrm{SO}(4)\to \mathrm{SU}(2)_+/\mathbb{Z}_2 \cong \mathrm{SO}(3)$. The kernel is $\mathrm{SU}(2)_-
< \mathrm{SO}(4)$. This gives the fibration 
 $$
 \frac{\mathrm{SO}(4)}{\mathrm{SU}(2)_-} \cong \mathrm{SO}(3) 
{\color{blue} \, \hookrightarrow \,}
S=\frac{{\mathrm{SO}}(n+4)}{{\mathrm{SO}}(n)\times {\mathrm{SU}}(2)_-}
 \, \longrightarrow \,
\widetilde{\mathbb{G}r}_4({\mathbb{R}}^{n+4})=\frac{{\mathrm{SO}}(n+4)}{{\mathrm{SO}}(n)\times {\mathrm{SO}}(4)},
 $$
that determines the $3$-Sasakian manifold $S$. 

To compute $p_1(S)$, consider the 
universal real oriented rank $4$-bundle $V \to \widetilde{\mathbb{G}r}_4({\mathbb{R}}^{n+4})$. By the description
of the cohomology (\ref{cohom:realgrasseven}), the roots corresponding to $V$ are $y_1,y_2$. 
The two real rank $3$-bundles $\bigwedge^2_+V, \bigwedge^2_-V$ are $\mathrm{SO}(3)=
\mathrm{SU}(2)/\mathbb{Z}_2$-bundles, associated to the two morphisms
$\mathrm{SO}(4)\to \mathrm{SU}(2)_\pm/\mathbb{Z}_2$. The roots of $\bigwedge^2_+V$ are $1,y_1+y_2, - y_1-y_2$,
and  the roots of $\bigwedge^2_-V$ are $1,y_1-y_2, - y_1+y_2$. Let us see this: we take $V=L_1\oplus L_2$, where
$L_1, L_2$ are $\SO(2)=\U(1)$-bundles with $y_1=c_1(L_1)$, $y_2=c_1(L_2)$. Then $\bigwedge^2_+\ox \CC=\bigwedge^{2,0}\oplus
\CC\omega \oplus \bigwedge^{0,2}\cong \CC \oplus (L_1\ox L_2) \oplus (\overline L_1\ox \overline L_2)$. The other case is similar.

The bundle associated to $S$ is $\bigwedge^2_+V$. So
 $$
 [\Omega]=p_1(S)=p_1(\bigwedge\nolimits^2_+V)= - c_2(\bigwedge\nolimits^2_+V\ox \CC)= (y_1+y_2)^2 = l+2x,
 $$
since $l=y^2_1+y^2_2$ and $x=y_1y_2$. Note that $p_1(\bigwedge^2_-V)= (y_1-y_2)^2=l-2x$, and that the
cohomology $H^*(\widetilde{\mathbb{G}r}_4({\mathbb{R}}^{n+4}))$ is invariant under 
$x \mapsto -x$.
\end{proof}

The following result will be useful for our purposes.

\begin{lemma}\label{lem:linearfactorsigma}
Consider the cohomology class $\tilde\sigma_r$ of degree $2r$ $(r \geq 0),$ defined by \eqref{eq:sigmaireal}, 
as a polynomial $\tilde\sigma_r=\tilde\sigma_r(l,x)$ in $l$ and $x$, 
and let $a,b\in \mathbb{Q}$, $(a,b)\neq (0,0)$.  Then
$\tilde\sigma_r$ factors through $al+bx$ if and only if one of the following
statements holds:
\begin{enumerate}
\item $r$ is odd and $b=0$;
\item $r\equiv2\pmod3$ and $|a|=|b|$.
\end{enumerate}
\end{lemma}
\begin{proof}
In order to determine the linear factors  
$al+bx$ of $\widetilde\sigma_r(l,x)$ for any $r$, we consider the generating function $\sum_{r=0}^\infty \widetilde\sigma_r(l,x)t^r$.
The recursive definition~\eqref{eq:sigmaireal} of $\widetilde\sigma_r(l,x)$ can be rewritten as the equality $(1+lt+x^2t^2)(\widetilde\sigma_0+\widetilde\sigma_1t+\widetilde\sigma_2 t^2+\cdots)=1$, so
\begin{equation}\label{eq:ratfunction3}
\sum_{r=0}^\infty \widetilde\sigma_r(l,x)t^r=\frac1{1+lt+x^2t^2}.
\end{equation}
Next, we are going to calculate the values $a,b\in \mathbb{Q}$ such that $al+bx$ divides $\widetilde\sigma_r(l,x)$.
First, let us consider the case $b=0$. Clearly, $al$ divides $\widetilde\sigma_r(l,x)$ if and only if $\widetilde\sigma_r(0,x)=0$. 
Then, the expansion series on $x\,t$ of~\eqref{eq:ratfunction3} becomes
\[
\sum_{r=0}^\infty \widetilde\sigma_r(0,x)t^r=\frac1{1+(xt)^2}=\sum_{r=0}^\infty (-1)^r (xt)^{2r}.
\]
Hence
\[
\widetilde\sigma_r(0,x)=\begin{cases}
(-1)^{r/2} x^r &\text{if $r$ is even,}\\
0&\text{if $r$ is odd.}
\end{cases}
\]
Therefore, $al$ divides $\widetilde\sigma_r(l,x)$ if and only if $r$ is odd. This is the condition~(1) of the statement of this Lemma.

Assume $b\neq 0$. Let $\lambda=\tfrac{a}{b}\in \mathbb{Q}$. 
Clearly, $a l+bx$ divides $\widetilde\sigma_r(l,x)$ if and only if $\widetilde\sigma_r(l,-\lambda l)=0$. 
We will determine the values of $\lambda\in\mathbb{Q}$ such
that $\widetilde\sigma_r(l,-\lambda l)=0$. First, we deal with $\lambda=0,\pm\frac12$.
\begin{itemize}
\item If $\lambda=0$, the expansion series on $lt$ of~\eqref{eq:ratfunction3} becomes
\begin{equation*}
\sum_{r=0}^\infty \widetilde\sigma_r(l,0)t^r =\frac1{1+lt}=\sum_{r=0}^r (-1)^rl^rt^r, 
\end{equation*} 
which implies that $\sigma_r(l,0)=(-1)^rl^r$, and so $bx$ does not divide $\widetilde\sigma_r(l,x)$ for any $r\geq0$. 

\item If $\lambda=\pm\tfrac12$, the expansion series on $lt$ of~\eqref{eq:ratfunction3} yields
\begin{equation*}
\sum_{r=0}^\infty \widetilde\sigma_r(l,\pm\tfrac12l)t^r=\frac1{1+lt+\tfrac14(lt)^2}=\frac1{(1+\tfrac12lt)^2}=\sum_{r=0}^\infty (-1)^r\frac{r+1}{2^r}(lt)^r.
\end{equation*}
This implies that $\widetilde\sigma_r(l,\pm\tfrac12l)=(-1)^r\frac{r+1}{2^r}l^r$. Thus, if $2|a|=|b|$, then $al+bx$ 
does not divide $\widetilde\sigma_r(l,x)$ for any $r\geq0$. 
\end{itemize}

Now, consider $\lambda\neq 0,\pm\tfrac12$. 
As before, we compute an explicit expression of $\widetilde\sigma_r(l,-\lambda l)$ by means of~\eqref{eq:ratfunction3}. 
We notice that \eqref{eq:ratfunction3} can be written as follows
\[
\sum_{r=0}^\infty \widetilde\sigma_r(l,-\lambda l)t^r=\frac1{1+lt+\lambda^2 (lt)^2}=\frac1{(\alpha-\beta)\lambda^2}\left(\frac1{lt-\alpha}-\frac1{lt-\beta}\right),
\]
where $\alpha,\beta\in\mathbb{C}$ such that
\begin{equation}\label{eq:alphabetaR}
\alpha+\beta=-\frac1{\lambda^2}\text{ and }\alpha\cdot\beta=\frac1{\lambda^2}.
\end{equation}
This can be done because we are taking $\lambda\neq 0,\pm\frac12$.
Moreover, by \eqref{eq:alphabetaR}, we have $\alpha-\beta=0$ if and only if $\lambda=\pm\frac12$.
We write each fraction as a geometric power series and arrange this as power series on $lt$:  
\begin{align*}
\sum_{r=0}^\infty \widetilde\sigma_r(l,-\lambda l)t^r&=\frac1{(\alpha-\beta)\lambda^2}\left(\frac1{lt-\alpha}-\frac1{lt-\beta}\right)\\
&=\frac1{(\alpha-\beta)\lambda^2}\left({1\over \beta}\sum_{r=0}^\infty \left(\frac{lt}{\beta}\right)^n-{1\over \alpha}\sum_{r=0}^\infty \left(\frac{lt}{\alpha}\right)^r\right)\\
&=\sum_{r=0}^\infty \frac1{(\alpha-\beta)\lambda^2} \left(\frac1{\beta^{r+1}}-\frac1{\alpha^{r+1}}\right)(lt)^r\\
&=\sum_{r=0}^\infty \frac{(\alpha/\beta)^{r+1}-1}{\alpha^{r+1}(\alpha-\beta)\lambda^2}l^rt^r.
\end{align*}
Hence, 
\begin{equation}\label{eq:sigmagenlambdaR}
\widetilde\sigma_r(l,-\lambda l)=\frac{(\alpha/\beta)^{r+1}-1}{\alpha^{r+1}(\alpha-\beta)\lambda^2}l^r,
\end{equation}
and $\widetilde\sigma_r(l,-\lambda l)=0$ if and only if $\alpha/\beta$ is a $(r+1)$th root of unity. Note that once we fix a value for $\alpha/\beta$, the equations~\eqref{eq:alphabetaR} determine $\lambda\in \mathbb{Q}$ up to sign. Then, we look for $\alpha/\beta$ which comes from $\alpha,\beta\in\mathbb{C}$ satisfying~\eqref{eq:alphabetaR}, and $k\geq1$ such that $\alpha/\beta$ is a primitive $k$th root of unity. 
We distinguish two possibilities, namely $\alpha/\beta\in\mathbb{Q}$ and 
$\alpha/\beta\not\in\mathbb{Q}$.

Assume $\alpha/\beta\in\mathbb{Q}$. In such a case, $\alpha/\beta$ is a root of unity if and only if $\alpha/\beta=\pm1$. 
\begin{itemize}
\item  $\alpha/\beta=1$ is equivalent to $\alpha-\beta=0$, which corresponds to the case $\lambda=\pm\tfrac12$ discussed previously;
\item $\alpha/\beta=-1$ is equivalent to $\alpha+\beta=0$. But this contradicts the first equality in~\eqref{eq:alphabetaR}.
\end{itemize}

If $\alpha/\beta\not\in\mathbb{Q}$, we check that $\alpha/\beta$ is a primitive $k$th root of unity by means of its minimal polynomial over $\mathbb{Q}$. More concretely, by uniqueness of the minimal polynomials, $\alpha/\beta$ is a primitive $k$th root of unity if and only its minimal polynomial is equal to the $k$th cyclotomic polynomial $\Phi_k(X)$. Therefore, $\alpha/\beta$ is a $(r+1)$th root of unity if and only if its minimal polynomial is equal to a $k$th cyclotomic polynomial for some $k$ divisor of $r+1$ or, equivalently, $r+1\equiv 0\pmod{k}$. 

Now, since $\alpha/\beta\not\in\mathbb{Q}$, its minimal polynomial over $\mathbb{Q}$ is of degree greater than or equal to~$2$. 
Taking into account~\eqref{eq:alphabetaR} we have that  
\[
-\left(\frac\alpha\beta+\frac\beta\alpha\right) =  -\frac{(\alpha+\beta)^2-2\alpha\beta}{\alpha\beta}=-\frac1{\lambda^2}+2,\qquad 
\frac\alpha\beta\cdot\frac\beta\alpha=1
\]
are rational numbers, and so $\alpha/\beta$ is a root of the polynomial $Q_\lambda(X)=X^2+\left(2-\tfrac1{\lambda^2}\right)X+1$. 
Thus $Q_\lambda(X)$ must be the minimal polynomial of $\alpha/\beta$
because it is of the lowest degree and $\mathbb{Q}$-irreducible. Now, let us check when $Q_\lambda(X)$ is equal to a cyclotomic polynomial. Since $\deg Q_
\lambda(X)=2$, the only possibilities are the cyclotomic polynomials $\Phi_k(X)$ of degree~$2$, that is, $\Phi_k(X)$ 
with $k=3,4$ and~$6$. We discuss each of these three cases:
\begin{itemize}
\item For $k=3$, the equality $Q_\lambda(X)=\Phi_3(X)=X^2+X+1$ 
implies $\frac{a}{b}=\lambda=1$ or $\frac{a}{b}=\lambda=-1$.
This means that $\alpha/\beta$ is the $3$rd root of unity and, by~\eqref{eq:sigmagenlambdaR}, $\widetilde\sigma_r(l,-\lambda l)=0$, for $r+1\equiv0\pmod3$ and $\lambda=\pm1$. Therefore, 
 if $|a|=|b|$, $al+bx$ divides $\widetilde\sigma_r(l,x)$ when $r\equiv 2\pmod3$, which is the condition~(2) of the statement of this Lemma.
\item For $k=4$, the equality $Q_\lambda(X)=\Phi_4(X)=X^2+1$ implies $\lambda=\tfrac{\sqrt2}2\not\in\mathbb{Q}$. But this is not possible because $\lambda$ must be a rational number.
\item For $k=6$, the equality $Q_\lambda(X)=\Phi_6(X)=X^2-X+1$ implies $\lambda=-\tfrac{\sqrt{3}}3\not\in\mathbb{Q}$,  which is not possible because $\lambda$ must be a rational number.
\end{itemize}
Finally, for the remaining values of $\lambda\in\mathbb{Q}$ (that is, $\lambda\neq-1$, and $\lambda\neq0,\pm\tfrac12$ which were discussed previously as specific cases), the minimal polynomial $Q_\lambda(X)$ of $\alpha/\beta$ is 
not a cyclotomic polynomial, and so $\alpha/\beta$ is not a root of unity. By~\eqref{eq:sigmagenlambdaR}, $\widetilde\sigma_r(l,-\lambda l)\neq0$,
 for any $r\geq0$. Therefore, $al+bx$ with $a/b=\lambda$ does not divide $\widetilde\sigma_r(l,x)$ for $r\geq0$.
\end{proof}

 To prove the formality of ${\mathrm{SO}}(n+4)\,/({\mathrm{SO}}(n)\times {\mathrm{Sp}}(1))$ with $n=2m\geq 4$
we study firstly the case when $n=2m\geq 6$.

\begin{theorem}\label{thm:prop:GrRevenformal}
 Consider  the principal $F$-fiber bundle 
$F \,\rightarrow \,S \,\rightarrow\, \widetilde{\mathbb{G}r}_4(\RR^{n+4})$ with $n=2m\geq 6$,
$F=\SU(2)$ or $F=\SO(3)$ and Euler class $e(S)=a l + bx$,  where $a, b \in\mathbb{Q}$.
Then $S$ is not formal if and only if one of the following statements is satisfied:
\begin{enumerate}
\item  $m$ is odd, $a\neq 0$ and $b=0$;
\item $m\equiv 2\pmod3$ and $|a|=|b|\neq 0$.
\end{enumerate}
In particular, for $n=2m\geq 6$, the homogeneous $3$-Sasakian manifold 
${\mathrm{SO}}(n+4)\,/({\mathrm{SO}}(n)\times {\mathrm{Sp}}(1))$ is formal.
\end{theorem}

\begin{proof}
 Since $\widetilde{\mathbb{G}r}_4(\RR^{n+4})$ is a compact positive quaternionic K\"ahler manifold, 
$\widetilde{\mathbb{G}r}_4(\RR^{n+4})$ is simply connected \cite{Salamon1}. Then, according with section \ref{min-models}, 
the fibre bundle $F \rightarrow \, S \,\rightarrow\, \widetilde{\mathbb{G}r}_4(\RR^{n+4})$, with $F=\SU(2)$ or $F=\SO(3)$ and Euler class
$a l + b x$,  is a rational $S^3$-fibration. Thus ~\cite{RS},
 if $(\mathcal{A},d_{\mathcal{A}})$ is a model of
$\widetilde{\mathbb{G}r}_4(\RR^{n+4})$, a model of $S$ is 
$(\mathcal{A}\otimes \bigwedge(u),d)$ with $|u|=3$,
$d|_{\mathcal{A}}=d_{\mathcal{A}}$ and $du=al+bx$.
On the other hand, $\widetilde{\mathbb{G}r}_4(\RR^{n+4})$ is formal by
Theorem~\ref{th:formalquat}, and hence a model of this space is
$(H^\ast(\widetilde{\mathbb{G}r}_4(\RR^{n+4})),0)$. 
Therefore, a model of $S$ is the differential graded algebra
\begin{equation} \label{model: S1}
\Big(H^\ast(\widetilde{\mathbb{G}r}_4(\RR^{n+4}))\otimes \bigwedge(u),d\Big),
\end{equation}
where $|u|=3$ and $du=al+bx$.

We show first that if one of the conditions $(1)$ or $(2)$ stated in this Theorem is satisfied, then $S$ is non-formal.
By Lemma \ref{lemm:massey-models},
we know that Massey products on a manifold $M$ can be computed
by using any model for $M$. Let us prove that the triple Massey product $\langle x, z, z\rangle$ is defined 
on $S$ and it is non-zero by using 
the model of $S$ given by \eqref{model: S1}.
 By \eqref{cohom:realgrasseven} and \eqref{model: S1},  $x\,z=0$ on $S$ because $x\,z=0$ in $H^*\big(\widetilde{\mathbb{G}r}_4(\RR^{n+4})\big)$. 
Moreover, by Lemma ~\ref{lem:linearfactorsigma}, each of conditions $(1)$ and $(2)$ 
implies that $\widetilde\sigma_m=(al+bx)\,\tau$, for some non-zero cohomology 
class $\tau$ of degree $4m-4$ on $\widetilde{\mathbb{G}r}_4(\RR^{n+4})$
such that $\tau$ only depends of $l$ and $x$. Thus, taking into account  \eqref{model: S1},
 $\widetilde\sigma_m=d(u\,\tau)$ on $S$.
So, by \eqref{cohom:realgrasseven}, $z^2=d(\tau u)$ on  $S$. Therefore, the triple Massey product $\langle x, z, z\rangle$ is defined on $S$ and 
$$
\langle x, z, z\rangle= x\,\tau\, u.
$$

Let us see that $x\,\tau\, u$ is not exact. The element $x\,\tau\, u$ is
exact in the model \eqref{model: S1} of $S$ if $x\,\tau=0$ in
$H^\ast(\widetilde{\mathbb{G}r}_4({\mathbb{R}}^{n+4}))$, because the image
of the differential map is contained in
$H^\ast(\widetilde{\mathbb{G}r}_4({\mathbb{R}}^{n+4}))$. But
$x\,\tau=0$ in $H^\ast(\widetilde{\mathbb{G}r}_4({\mathbb{R}}^{n+4}))$ if 
and only if $x\,\tau$ belongs to the ideal
$(x\,z, z^2- \widetilde\sigma_{m}, \widetilde\sigma_{m+1})$. Let us to show
that this is not possible. Since
$4m=|x\,\tau|<|\widetilde\sigma_{m+1}|=4m+4$, $x\,\tau$ would be a
combination of $xz$ and $z^2-\widetilde\sigma_m$. We know that 
$x\,\tau$ is a combination of $l$ and $x$ because $\tau$ only depends of $l$ and $x$.  
On the other hand, any combination of $x\,z$ and
$z^2-\widetilde\sigma_m$ that is independent of $z$, it must be of degree greater than
or equal $4m+4$. Hence, $x\,\tau$ does not belong to the ideal $(x\,z, z^2-
\widetilde\sigma_{m}, \widetilde\sigma_{m+1})$ and thus $x\,\tau$ is non-zero. 
Therefore, $\langle x, z, z\rangle$ is a non-trivial Massey product on $S$ and, by Lemma \ref{lem:criterio1}, $S$ is non-formal.

Conversely, we must prove that if $S$ is non-formal, then one of conditions $(1)$ or $(2)$ is satisfied.
But this is equivalent to prove that $S$ is formal if $a=0$ or $a\not=0$ but $\widetilde\sigma_{m}$ on 
$S$ is non-zero (that is, $\widetilde\sigma_{m}$ does not factorize by $al + bx$). We study each of these two cases separately.
\begin{itemize}
 \item Suppose that  $a\neq0$ and $\widetilde{\sigma}_m\neq0$. Clearly, $l = -(b\,/a)\,x$ in $H^*(S)$ since $a\not=0$ and, by \eqref{model: S1}, 
 $du=al+bx$. 
 Then, \eqref{eq:realsigma} implies that the cohomology class $\widetilde{\sigma}_m$ on $S$ has the following expression
 \begin{align*}
 & \widetilde{\sigma}_m =  x^{m} \sum_{k=0}^{\lfloor m/2\rfloor}(-1)^{m+k}\binom{m-k}{k} \left(-\frac{b}{a}\right)^{m-k}. 
 \end{align*}
Thus the cohomology class $z^2$ on $S$ is a multiple of the cohomology class $x^m$ because, by \eqref{cohom:realgrasseven},
$z^2$ and $\widetilde{\sigma}_m$ are the same cohomology class on $\widetilde{\mathbb{G}r}_4({\mathbb{R}}^{n+4})$ and so on $S$. 
(Note that if $m$ is odd, the condition $\widetilde\sigma_{m}\not=0$ implies $b\not=0$, but if $m$ is even then $b$ may be $0$.) 
Moreover, the class $x^{m+1}$ is zero on $S$ since it is a multiple of  
$z(xz)-x(z^2-\bar\sigma_m)$, which is zero on $\widetilde{\mathbb{G}r}_4(\RR^{n+4})$ by (\ref{cohom:realgrasseven}).
Taking into account \eqref{cohom:realgrasseven} and the model of $S$ given by \eqref{model: S1}, we have that
the non-zero cohomology groups of $S$ up to the degree $2n+1$ are
  \begin{align*}
& H^{0}(S)=\langle 1\rangle\,, \qquad H^{4k}(S)=\langle x^k\rangle, \quad 0 \leq k \leq m,  \quad \text{ and} \quad  4k\neq 2m,\\
&H^{2m}(S)=\begin{cases}\langle x^{m/2}, z\rangle,&\text{ if $m$
 is even,}\\ \langle z\rangle,&\text{ if $m$ is odd.} \end{cases}
  \end{align*}
Therefore, the minimal model of $S$ must be a differential graded algebra
$(\bigwedge V,d)$, being $\bigwedge V$ the free algebra of the form
$\bigwedge V=\bigwedge(a_4, a_{2m}, v_{2m+3}, v_{4m-1})\otimes \bigwedge V^{\geq (2n+2)}$, where $|a_i|=i$, for $i=4, 2m$, 
$|v_{j}|=j$, for $j= 2m+3, 4m-1$, and $d$ is defined by $da_i=0$, $dv_{2m+3}=a_{4}\,a_{2m}$ and $dv_{4m-1}=a_{4}^{m}-a_{2m}^2$. 
Let us prove that $S$ is $(4m+1)$-formal. Any element $\alpha \in I(N^{\leq 4m+1})$ is of the form $\alpha=P_1\, v_{2m+3}+P_2\, v_{4m-1}+P_3 \, v_{2m+3}\, v_{4m-1}$, where
$P_1,P_2,P_3\in \bigwedge(a_4, a_{2m})$. If $\alpha$ is a closed element
of even degree, then $\alpha=P_3 \,  v_{2m+3}\, v_{4m-1}$. But 
$d\alpha=0$ implies  $P_3=0$, and so $\alpha$ is trivially exact.  If
$\alpha$ is a closed element of odd degree, then
$\alpha=P_1\, v_{2m+3}+P_2\, v_{4m-1}$. The condition $d\alpha=0$ implies $P_1=P \, (a_{2m}^2-a_{4}^m)$ and
$P_2=-P \, a_{4} \, a_{2m}$, for some $P\in\bigwedge(a_4, a_{2m})$. Hence,
$\alpha=d(P  \, v_{2m+3}\, v_{4m-1})$ is exact. This proves that $S$ is $(4m+1)$-formal and, by Theorem~\ref{fm2:criterio2}, $S$ is formal.

 \item Consider the case $a=0$, so $du=bx$. If $b=0$, then $S$ is clearly formal because
$S$ is rationally equivalent to $S^3\times \widetilde{\mathbb{G}r}_4(\RR^{n+4})$, which
is formal being the product of two formal manifolds.

Suppose now that $a=0$ and $b\neq 0$. The condition $du = bx$ and
 \eqref{eq:realsigma} imply that the element $\widetilde{\sigma}_m$ of the model of $S$ given by \eqref{model: S1} 
 has the following expression
 \begin{align*}
 & \widetilde{\sigma}_m =  (-1)^{m} l^{m} + \text{ elements which are exact in}\,  \eqref{model: S1}.
 \end{align*}
Thus the cohomology class $z^2$ on $S$ is  $\pm\,l^m$ because, by \eqref{cohom:realgrasseven},
$z^2$ and $\widetilde{\sigma}_m$ are the same cohomology class on $\widetilde{\mathbb{G}r}_4({\mathbb{R}}^{n+4})$.
Similarly, using again  \eqref{eq:realsigma}, the condition $du = bx$ on the model of $S$, and taking into account that, 
by \eqref{cohom:realgrasseven},  $\widetilde{\sigma}_{m+1}=0$ on 
$\widetilde{\mathbb{G}r}_4({\mathbb{R}}^{n+4})$, one can check that $l^{m+1}$ is the zero on $S$.
the non-zero cohomology groups of even degree of $S$ up to degree $(2n+1)$ are
 \begin{align*}
&H^{0}(S) =\langle1 \rangle\,, \qquad H^{4j}(S)=\langle l^j\rangle\,, \,\,\,\, 1 \leq j \leq m,  \\
&H^{2m+4k}(S)=\langle l^k\,z\rangle\,, \,\,\,\, 0 \leq k \leq {\lfloor m/2\rfloor}\,,
 \end{align*}
 if $m$ is odd; and
 \begin{align*}
 &H^{0}(S) =\langle1 \rangle\,, \qquad H^{4k}(S)=\begin{cases}
  \langle l^k\rangle, &\text{ if }0\leq k<\tfrac{m}2,\\
   \langle l^k,l^{k-(m/2)}z\rangle,&\text{ if }\tfrac{m}2\leq k\leq m,\\
  \end{cases}
   \end{align*}
if $m$ is even. But in both cases, 
$S$ has non-zero cohomology groups of odd degree up to degree $2n+1$, namely
$$
H^{2m+3+4k}(S)=\langle l^kz\,u\rangle,
$$ 
for $0\leq k < \lfloor m/2\rfloor$. Therefore, the minimal model of $S$ must be a differential graded algebra $(\bigwedge V, d)$, being
$\bigwedge V$ the free algebra of the form $\bigwedge V=\bigwedge(a_{4}, a_{2m}, a_{2m+3}, v_{4m-1})\otimes V^{\geq (2n+2)}$, 
where $|a_i|=i$ $(i= 4, 2m, 2m+3)$, $|v_{4m-1}|=4m-1$, 
$da_{4}=da_{2m}=da_{2m+3}=0$ and
$dv_{4m-1}=a_{2m}^2-a_{4}^m$. Clearly, $S$ is $(4m+1)$-formal, and hence it is formal by
Theorem~\ref{fm2:criterio2}.
\end{itemize}

Let us consider the $3$-Sasakian homogeneous space 
$S={\mathrm{SO}}(n+4)\,/({\mathrm{SO}}(n)\times {\mathrm{Sp}}(1))$ with $n=2m\geq 6$. 
From \eqref{eqn:eulerclassSO(3)}, Theorem \ref{class-homog-3-sasaki} and Proposition \ref{prop:formgrassr}, 
$S$ is the $\SO(3)$-bundle
$\SO(3)\rightarrow \, S \,\rightarrow\, \widetilde{\mathbb{G}r}_4(\RR^{n+4})$ with Euler
class $-\frac14(l+2x)$. Hence, none
of the conditions $(1)$ and $(2)$ stated in this Theorem is satisfied because
now $a=-\frac{1}{4} \not =-\frac{1}{2}=b$. Thus, $S$ is formal.
Note that the minimal model of $S$ is the minimal model previously described for 
$a\neq0$ and $\widetilde{\sigma}_m\neq0$.
\end{proof}

 Now, in order to show that that the $3$-Sasakian homogeneous space 
${\mathrm{SO}}(n+4)\,/({\mathrm{SO}}(n)\times {\mathrm{Sp}}(1))$ is formal,
for $n=2m\geq 4$, only remains to prove that this space is formal for $n=4$.
For this, note that
\eqref{cohom:realgrasseven} implies that, for $n=4$, the cohomology ring  of $\widetilde{\mathbb{G}r}_4(\RR^8)$ is
\begin{equation}\label{cohom:GrR8}
	H^\ast(\widetilde{\mathbb{G}r}_4(\RR^8))=\mathbb{Q}[l,x,z]/(xz,z^2-l^2+x^2,l^3-2lx^2),
\end{equation}
where $|x|=|y|=|z|=4$.

\begin{theorem}\label{thm:prop:GrR8formal}
Let $F=\SO(3), \SU(2)$. Consider a fiber bundle $F\to S\to \widetilde{\mathbb{G}r}_4(\RR^{8})$ with
Euler class $e(S)=a l + bx+cz$,  where $a, b, c\in\mathbb{Q}$. Then,  $S$ is not formal if and only if one of the following statements is satisfied:
 \begin{enumerate}
 \item $|a|=|b|\neq0$ and $|a|\neq |c|$,
 \item $|a|=|c|\neq0$ and $|a|\neq |b|$.
 \end{enumerate}
 In particular, the homogeneous $3$-Sasakian manifold ${\mathrm{SO}}(8)\,/({\mathrm{SO}}(4)\times {\mathrm{Sp}}(1))$ 
is formal.
\end{theorem}

\begin{proof}
Since $\widetilde{\mathbb{G}r}_4(\RR^{8})$ is simply connected and formal, proceeding as in the
proof of Theorem~\ref{thm:prop:GrRevenformal}, a model of $S$ is
\begin{equation}\label{eq:modelGrR8}
	\left(H^\ast(\widetilde{\mathbb{G}r}_4(\RR^{8}))\otimes
		\bigwedge(u),d\right),
\end{equation}
where $|u|=3$ and $du=al+bx+cz\in H^4(\widetilde{\mathbb{G}r}_4(\RR^{8}))$. 

In order to prove that if one of the conditions $(1)$ or $(2)$ stated in this Theorem is satisfied, then $S$ is non-formal, 
we consider first a new basis of the cohomology ring $H^\ast(\mathbb{G}r_4(\mathbb{R}^8))$.
For this we proceed as follows. In any of the cases $(1)$ or $(2)$, we have 
$(b,c)\neq (0,0)$. We can assume that $a,b,c\geq0$, and so $b+c>0$ since $(b,c)\neq (0,0)$. 
(In fact, it is sufficient to change $l$ by $-l$ if $a <0$, or $x$ by $-x$ if $b<0$, or $z$ by $-z$ if $c<0$.)
Now we define the basis 
$\xi_0,\xi_1,\xi_2$ of $H^{4}(\mathbb{G}r_4(\mathbb{R}^8))$ by
$$
\xi_0=al+bx+cz, \qquad \xi_1=l-x+z, \qquad \xi_2=-l-x+z.
$$
Hence, $\xi_0,\xi_1$ and $\xi_2$ generate the cohomology ring $H^\ast(\mathbb{G}r_4(\mathbb{R}^8))$,  and
the relations given by  \eqref{cohom:GrR8} can be written in terms of $\xi_0$, $\xi_1$ and $\xi_2$  as 
\begin{equation}
\begin{gathered}\label{eq:exactclasses}
\xi_1\xi_2, \qquad \frac{1}{4(b+c)^2}\left((a-b)(a+c)\xi_1^2+(a+b)(a-c)\xi_2^2\right)+\beta_4 \xi_0,\quad \text{ and }\\
\frac{2(a-c)^2-(b+c)^2}{8(b+c)^2}\left(\xi_1^3-\xi_2^3\right)+\beta_8\xi_0,
\end{gathered}
\end{equation}
where $\beta_4$ and $\beta_8$ are elements of the cohomology ring~\eqref{cohom:GrR8} of degree $|\beta_i|=i$, with 
\begin{equation}\label{eq:beta}
\beta_4=\frac1{(b+c)^2}\xi_0+\frac{-2a+b-c}{2(b+c)^2}\xi_1+\frac{2a+b-c}{2(b+c)^2}\xi_2\,.
\end{equation}

 According with the model of $S$ given by \eqref{eq:modelGrR8}, we have now $du=\xi_0$. 
Then, from  \eqref{eq:exactclasses}, $(a-b)(a+c)\xi_1^2+(a+b)(a-c)\xi_2^2$ and $\xi_1^3-\xi_2^3$ are exact in the model~\eqref{eq:modelGrR8}. 
 So, the non-zero cohomology groups of $S$ up to degree~12 are:
\begin{equation}\label{eq:H8GrR8}
	H^4(S)=\langle \xi_1,\xi_2\rangle, \,\, \qquad \,\,
 	H^8(S)=\begin{cases}
	    \langle \xi_1^2\rangle, &\text{ if $a\neq b$ and $a\neq c$},\\
 		\langle \xi_1^2\rangle,&\text{ if $a=b\neq 0$ and $a\neq c$},\\
 		\langle \xi_2^2\rangle,&\text{ if $a=c\neq 0$ and $a\neq b$},\\
 		\langle \xi_1^2,\xi_2^2\rangle,&\text{ if $a=b=c$ or $a=b=0$ or $a=c=0$},
\end{cases}
\end{equation}
   and 
  $$
 H^{12}(S)=\begin{cases}
 \langle \xi_1^3\rangle,& \text{ if $a=b=c$,  or $a=b=0$ or $a=c=0$}\\
 0,&\text{ otherwise.} \end{cases}
 $$
We may determine the odd degree cohomology groups by applying the Poincar\'e duality. 
If $a=b=c$ or $a=b=0$ or $a=c=0$, $H^7(S)$ is generated by the Poincar\'e dual $PD(\xi_1^3)$ of 
the non-zero cohomology class $\xi_1^3\in H^{12}(S)$.

To prove that $S$ is non-formal if the condition $(1)$ is satisfied, 
we apply Lemma~\ref{lemm:massey-models} to compute a non-trivial triple Massey product in the model~\eqref{eq:modelGrR8}.
Since $a,b,c\geq0$, the condition $(1)$ becomes $a=b\neq 0$ and $a\neq c$. 
From~\eqref{eq:exactclasses}, we have that $\xi_1\xi_2$ is zero. Moreover, the second class in~\eqref{eq:exactclasses} yields 
$\xi_2^2=-d\left(\frac{2(b+c)^2}{b(b-c)}\beta_4u\right)$ since $du=\xi_0$ on $S$. 
Therefore, the triple Massey product
$\langle \xi_2,\xi_2,\xi_1\rangle$ is defined and, using \eqref{eq:beta}, 
 we have
$$
 \langle \xi_2,\xi_2,\xi_1\rangle = -\frac{2(b+c)^2}{b(b-c)}\beta_4\,u\,\xi_2=\frac{1}{b(b-c)}\left(-2\xi_0-(3b-c)\xi_2+(b+c)\xi_1\right)\xi_1\,u,
$$
that is the class $\frac{(b+c)}{b(b-c)}\xi_1^2\,u\in H^{11}(S)$ since $\xi_{0}$ defines the zero cohomology class 
on $S$. Thus, the  triple Massey product $ \langle \xi_2,\xi_2,\xi_1\rangle$ is non-trivial because 
$b\not= 0$ and $a=b\neq c$. 
Hence $S$ is non-formal by Lemma~\ref{lem:criterio1}.

A similar argument shows that if condition $(2)$ is satisfied, then $S$ is non-formal. In fact, since $a,b,c\geq0$, condition $(2)$ is 
equivalent to $a=c\neq0$ and $a\neq b$.
Then, from \eqref{eq:exactclasses},
$\xi_1\xi_2$ is zero and $\xi_1^2=d\left(\tfrac{2(b+c)^2}{c(b-c)}\beta_4u\right)$ is exact in the model~\eqref{eq:modelGrR8}. 
Thus, the triple Massey product $\langle \xi_1,\xi_1,\xi_2\rangle$ is defined and by~\eqref{eq:beta} we have
\[
\langle \xi_1,\xi_1,\xi_2\rangle = \frac{2(b+c)^2}{c(b-c)}\beta_4\,u\,\xi_2=\frac1{c(b-c)}\left(2\xi_0+(b-3c)\xi_1+(b+c)\xi_2\right)\xi_2\,u.
 \]
This is the non-zero class $\frac{(b+c)}{c(b-c)}\xi_2^2\,u\in H^{11}(S)$. Hence $\langle \xi_1,\xi_1,\xi_2\rangle$ is non-trivial and, 
by Lemma~\ref{lem:criterio1}, $S$ is non-formal.

Now, we prove the converse statement, that is, if $S$ is non-formal, then one of the conditions $(1)$ or $(2)$ is satisfied. 
This is equivalent to prove that $S$ is formal if 
none of these conditions $(1)$ and $(2)$ is satisfied. But this happens in one of the following cases:
\begin{itemize}
\item If $a\neq b$ and $a\neq c$, the minimal model of $S$ is a differential graded algebra $(\bigwedge V,d)$, being $\bigwedge V=\bigwedge(a_4,b_4,v_7,w_7)\otimes \bigwedge V^{\geq 10}$ and $|a_4|=|b_4|=4$, $|v_7|=|w_7|=7$, and $da_4=0$, $db_4=0$, $dv_7=a_4b_4$, $dw_7=(a-b)(a+c)a^2_4+(a+b)(a-c)b^2_4$.  In order to prove the formality, let us consider a closed element $\alpha\in I(N^{\leq 9})$. Then $\alpha=P_1v_7+P_2w_7+P_3\, v_7 w_7$, where $P_1,P_2,P_3\in \bigwedge(a_4,b_4)$. If $\alpha$ is of even degree, $\alpha=P_3\, v_7 w_7$. Clearly, $d\alpha=0$ implies $P_3=0$, and hence $\alpha$ is trivially exact. If $\alpha$ is of odd degree, then $\alpha=P_1v+P_2w$. The condition $d\alpha=0$ implies that there exists $P\in \bigwedge(a_4,b_4)$ such that $P_1=P((a-b)(a+c)a^2_4+(a+b)(a-c)b^2_4)$ and $P_2=-P\, a_4 b_4$. Thus, $\alpha=d(P\, v_7 w_7)$ is exact. This proves the $9$-formality and therefore $S$ is formal by Theorem~\ref{fm2:criterio2}.

\item For the case $a=b=c$ or $a=b=0$ or $a=c=0$, 
recall that $H^{7}(S)$ is generated by $PD(\xi_1^3)$ the Poincar\'e dual of $\xi_1^3$. Then, the minimal model of $S$ is a differential graded algebra $(\bigwedge V,d)$, being 
$\bigwedge V=\bigwedge(a_4,b_4,c_7,v_7)\otimes \bigwedge V^{\geq 10}$ where $|a_4|=|b_4|=4$, $|c_7|=|v_7|=7$, 
and $da_4=db_4=dc_7=0$ and $dv_7=a_4b_4$. Clearly, $S$ is $9$-formal and, by Theorem~\ref{fm2:criterio2}, $S$ is formal.

\item For $b=c=0$. If $a=0$ then the fibration $S$ is rationally equivalent to $S^3\times \widetilde{\mathbb{G}r}_4(\mathbb{R}^8)$, which is formal being  the product of two formal manifolds. 
Then, assume that $a\neq 0$. Since $du = al$, $H^4(S)=\langle x,z\rangle$. The relations of~\eqref{cohom:GrR8} yield the equalities $xz=0$ and $x^2+z^2=\tfrac1{a}ldu$ in the model~\eqref{eq:modelGrR8}, so $H^8(S)=\langle x^2\rangle$ and $H^{4k}(S)=0$ for $k\geq3$. 
Then, the minimal model of $S$ is a differential graded algebra $(\bigwedge V,d)$, being $\bigwedge V$ the free algebra of the form $\bigwedge V=\bigwedge(a_4,b_4,v_7,w_7)\otimes \bigwedge V^{\geq 10}$, where $|a_4|=|b_4|=4$, $|v_7|=|w_7|=7$, and $da_4=0$, $db_4=0$, $dv_7=a_4b_4$, $dw_7=a^2_4+b^2_4$. To prove the
formality, take a closed element $\alpha\in I(N^{\leq 9})$. Then,
$\alpha=P_1v_7+P_2w_7+P_3\, v_7w_7$, where $P_1,P_2,P_3\in \bigwedge(a_4,b_4)$.
If $\alpha$ is of even degree, $\alpha=P_3\,v_7w_7$. The equality $d\alpha=0$
implies $P_3=0$, and hence $\alpha=0$ is exact. If $\alpha$ is of odd degree,
$\alpha=P_1v_7+P_2w_7$. Since $d\alpha=0$, there exists $P\in \bigwedge
(\alpha,\beta)$ such that $P_1=P(a^2_4+b^2_4)$ and $P_2=-P\,a_4b_4$. Thus, $\alpha=d(P\,v_7w_7)$
is exact. This proves that $S$ is $9$-formal and, by Theorem~\ref{fm2:criterio2}, $S$ is formal.
\end{itemize}

From \eqref{eqn:eulerclassSO(3)}, Theorem \ref{class-homog-3-sasaki} and Proposition \ref{prop:formgrassr}, 
we know that the $3$-Sasakian homogeneous space $S=\SO(8)/(\SO(4)\times \mathrm{Sp}(1))$ is 
the  $\SO(3)$-bundle $\SO(3)\to S\to \widetilde{\mathbb{G}r}_4(\mathbb{R}^8)$ with Euler class $-\frac14(l+2x)$. Then, 
$(a,b,c)=\left(-\tfrac14,-\tfrac12,0\right)$ and so none of the conditions~(1) and~(2) stated in the Theorem is satisfied. 
Thus, $S$ is formal. The minimal model of $S$ is the minimal model described before  for the case $a\neq b$ and $a\neq c$. 
\end{proof}

\subsection{$n$ is odd} If $n=2m+1\geq 3$, the cohomology of
$\widetilde{\mathbb{G}r}_4(\RR^{n+4})$ is 
 $$
 H^\ast(\widetilde{\mathbb{G}r}_4(\RR^{n+4}))=H^\ast(BT)^{W(\SO(n)\times
 \SO(4))}/H^{>0}(BT)^{W(\SO(n+4))},
 $$
where $BT$ is the classifying space of a maximal torus $T$ of $\SO(n+4)$ and
$W(G)$ denotes the Weyl group of a Lie group $G$. Then, the cohomology of
$\widetilde{\mathbb{G}r}_4(\RR^{n+4})$ are invariant polynomials of
$H^\ast(BT)=\mathbb{Q}[x_1,\ldots,x_{m+2}]$,  where $|x_i|=2$, for $1 \leq i \leq (n+2)$.
Denote by $y_1,y_2$ the
classes $x_{m+1},x_{m+2}$, respectively. On one hand,
$H^\ast(BT)^{W(\SO(n)\times \SO(4))}$ is generated by the symmetric polynomials
$\tilde\tau_k$ of $x_1^2,\ldots,x_{m}^2$ and the symmetric polynomials
$\tilde\sigma_k$ of $y_1^2,y_2^2$, and also by $y_1y_2$. On the other hand,
$H^{>0}(BT)^{W(\SO(n+4))}$ is
generated by the symmetric polynomials $\sigma_k$ of
$x_1^2,\ldots,x_{m}^2,y_1^2,y_2^2$. This gives the relation $\widetilde\sigma\cdot\tilde\tau=1$ between
$ \widetilde\sigma_k$ and $\tilde\tau_k$,  where
$ \widetilde\sigma=1+ \widetilde\sigma_1+ \widetilde\sigma_2+\ldots$ and
$ \widetilde\tau=1+\tilde\tau_1+ \widetilde\tau_2+\ldots$ Denote by $l$ and $x$ the classes $l=y_1^2+y_2^2$ and
$x=y_1\,y_2$. We have~\cite{Amann1}
\begin{equation}\label{cohom:GrRodd}
	H^\ast(\widetilde{\mathbb{G}r}_4(\RR^{n+4}))=\mathbb{Q}[l,x]/( \widetilde\sigma_{m+1}, \widetilde\sigma_{m+2}),
\end{equation} 
where $|l|=|x|=4$, and $\widetilde\sigma_r$ $(r \geq 0)$ is the cohomology class of degree $4r$ defined recursively as in
\eqref{eq:sigmaireal}. Note that $\widetilde\sigma_r$ also satisfies \eqref{eq:realsigma}.

The same proof as that given for Proposition \ref{prop:formgrassr} allows us to prove the following:
\begin{proposition}\label{prop:formgrassrODD}
Let $\Omega$ be  the quaternionic K\"ahler form on $\widetilde{\mathbb{G}r}_4({\mathbb{R}}^{n+4})$ with $n=2m+1\geq 3$.
Then, in terms of the generators $l$ and $x$ of $H^*(\widetilde{\mathbb{G}r}_4({\mathbb{R}}^{n+4}))$ given by \eqref{cohom:GrRodd},
the de Rham cohomology class $[\Omega]\in H^4(\widetilde{\mathbb{G}r}_4({\mathbb{R}}^{n+4}))$ is $[\Omega]=l+2x$.
\end{proposition}

\begin{theorem} \label{thm:prop:formalidadreal-final}
Consider a principal $F$-fiber bundle 
$F \,\rightarrow \,S \,\rightarrow\, \widetilde{\mathbb{G}r}_4(\RR^{n+4})$ with $n=2m+1\geq 3$,
$F=\SU(2)$ or $F=\SO(3)$ and Euler class $e(S)=a l + bx$, where $a, b \in\mathbb{Q}$.
Then, $S$ is formal. In particular, if $n=2m+1\geq 3$, the $3$-Sasakian homogeneous space 
${\mathrm{SO}}(n+4)\,/({\mathrm{SO}}(n)\times {\mathrm{Sp}}(1))$ is formal.
\end{theorem}

\begin{proof} 
Since $\widetilde{\mathbb{G}r}_4(\RR^{n+4})$ is simply connected and formal, proceeding as in the proof
of Theorem \ref{thm:prop:GrRevenformal}, we have that a model of $S$ is the differential graded algebra
\begin{equation} \label{model: S3}
(H^\ast(\widetilde{\mathbb{G}r}_4(\RR^{n+4}))\otimes \bigwedge(u),d),
\end{equation}
where $|u|=3$ and $du=al+bx$.

If $a=0=b$, then $S$ is rationally equivalent to $S^3\times \widetilde{\mathbb{G}r}_4(\RR^{n+4})$, which
is formal being the product of two formal manifolds.

Suppose now that $a\neq 0$ and $b=0$. Then, if $m$ is odd, using that $\widetilde{\sigma}_{m+1}=0$ on 
$\widetilde{\mathbb{G}r}_4(\RR^{n+4})$, from \eqref{eq:realsigma} we have
 \begin{align*}
\pm\, x^{m+1} + \text{ elements which are exact in}\,  \eqref{model: S3} =0
 \end{align*}
on $S$. This means that $x^{m+1}=0$ on $S$. Then, since $du=al$ by \eqref{model: S3},
the cohomology of $S$ up to the degree $2n+1=4m+3$ is $H^{0}(S)=\langle 1\rangle\,$, $H^{2i+1}(S) =0$, for  $0 \leq i \leq 4m+3$, 
and $H^{4j}(S)=\langle x^j\rangle$, for $1 \leq j \leq m$. 
Thus, the minimal model of $S$ must be a differential graded algebra 
$(\bigwedge V, d)$, being $\bigwedge V$ the free algebra of the form $\bigwedge V=\bigwedge(a_{4}, v_{4m+3})\otimes V^{\geq (4m+4)}$, 
where $|a_4|=4$, $|v_{4m+3}|=4m+3$, $da_{4}=0$ and $dv_{4m+3}=a_{4}^{m+1}$. 
Now, take $\alpha\in I(N^{\leq {4m+3}})$ a closed
element in $\bigwedge V$. Then $\alpha=a_{4}^{p} \,v_{4m+3}$ which is not closed, for any integer number $p\geq 1$.
So, according with Definition \ref{def:primera}, $S$ is $(4m+3)$-formal and, by Theorem
\ref{fm2:criterio2}, $S$ is formal. 

Moreover, if $a\neq 0$ and $b=0$, but $m$ is even, using that $\widetilde{\sigma}_{m+2}=0$ on 
$\widetilde{\mathbb{G}r}_4(\RR^{n+4})$, from \eqref{eq:realsigma} we have
 \begin{align*}
\pm\, x^{m+2} + \text{ elements which are exact in}\,  \eqref{model: S3} =0,
 \end{align*}
that is $x^{m+2}=0$ on $S$. 
Furthermore, by Lemma~\ref{lem:linearfactorsigma}, $\widetilde\sigma_{m+1}=\tau\cdot (al)$ for some non-zero class of 
degree $4m$ on $\mathbb{G}r_4(\RR^{n+4})$. Thus, taking into account that $d(\tau\,u)=\widetilde\sigma_{m+1}$ that is zero 
on $\mathbb{G}r_4(\RR^{n+4})$, $\tau\,u$ is a closed element on $S$.
 The cohomology of $S$ up to the degree $2n+1=4m+3$ is $H^{0}(S)=\langle 1\rangle\,$,
$H^{2i+1}(S) =0$, for  $0 \leq i \leq 2m$, $H^{4j}(S)=\langle x^j\rangle$, for $1 \leq j \leq m$, and $H^{4m+3}(S)=\langle \tau\,u\rangle$. 
Hence the minimal model of $S$ must be a differential graded algebra 
$(\bigwedge V, d)$, being $\bigwedge V$ the free algebra of the form $\bigwedge V=\bigwedge(a_{4}, a_{4m+3})\otimes V^{\geq (4m+4)}$, 
where $|a_4|=4$, $|a_{4m+3}|=4m+3$ and the differential is defined by $da_{4}=0=da_{4m+3}$. So $N^j=0$ for $j\leq (4m+3)$. Thus, according with Definition~\ref{def:primera}, 
the manifold $S$ is $(4m+3)$-formal and, by Theorem~\ref{fm2:criterio2}, $S$ is formal. 

If $a=0$ and $b\not=0$, proceeding as in the previous case we have that $l^{m+1}=0$ or $l^{m+1}\not=0$ but $l^{m+2}=0$ on $S$.
If $l^{m+1}=0$, then the minimal model of $S$ is the one given in the previous case when $x^{m+1}=0$, and so $S$ is formal. If
$l^{m+1}\not=0$ but $l^{m+2}=0$ on $S$, the minimal model of $S$ is the one given in the case $a\neq 0$, $b=0$ and $x^{m+2}=0$.
Thus $S$ is formal. 

Suppose that $a\not=0$ and $b\not=0$. Then,  $l=-{\frac{b}{a}} x$ on $S$ since $a\not=0$ and  $du=al+bx$ by \eqref{model: S3}.
Using that $\widetilde{\sigma}_{m+1}=0=\widetilde{\sigma}_{m+2}$ on 
$\widetilde{\mathbb{G}r}_4(\RR^{n+4})$, from \eqref{eq:realsigma} we have that  $x^{m+1}=0$ or $x^{m+1}\not=0$ but $x^{m+2}=0$ on $S$.
If $x^{m+1}=0$,  the minimal model of $S$ is the one given in the case $a\neq 0$, $b=0$ and $x^{m+1}=0$, and so $S$ is formal. When
 $x^{m+1}\not=0$ but $x^{m+2}=0$, the minimal model of $S$ is the one given in the case $a\neq 0$, $b=0$ and $x^{m+2}=0$.
 Therefore, $S$ is formal.
 
Now consider the $3$-Sasakian homogeneous space 
$S={\mathrm{SO}}(n+4)\,/({\mathrm{SO}}(n)\times {\mathrm{Sp}}(1))$ with $n=2m+1\geq 3$. 
From \eqref{eqn:eulerclassSO(3)}, Theorem \ref{class-homog-3-sasaki} and Proposition \ref{prop:formgrassrODD}, 
$S$ is the $\SO(3)$-bundle
$\SO(3)\rightarrow \, S \,\rightarrow\, \widetilde{\mathbb{G}r}_4(\RR^{n+4})$ with Euler
class $-\frac14(l+2x)$, and so $S$ is formal. Indeed, a minimal model of $S$ is the minimal model previously described for 
$a\neq0$ and $b\neq0$.
\end{proof}

\section{$\SU(2)$ and $\SO(3)$-bundles over the exceptional Wolf spaces} \label{sect:exceptional}

Here we prove that the exceptional $3$-Sasakian homogeneous spaces appearing in  Theorem \ref{class-homog-3-sasaki}
are all formal. They are principal $\SO(3)$-bundles over the exceptional Wolf spaces. We study each of these spaces separately.
 We also show that the total space of $\SU(2)$ and $\SO(3)$-bundles over the exceptional Wolf spaces are formal.

\subsection{The Wolf space $GI$}
First we consider the $8$-dimensional homogeneous quaternionic K\"ahler manifold 
  $$
 {{GI}}\,=\,\frac{{\mathrm{G}}_2}{{\mathrm{SO}}(4)}.
  $$
The rational cohomology ring of ${\mathrm{G}}_2\,/ {\mathrm{SO}}(4)$ is given by \cite{Borel-Hirz, Ishitoya-Toda}
 \begin{equation}\label{cohom:GI}
 H^*(GI) =\mathbb{Q}[x]/(x^3), 
 \end{equation}
where $x$ has degree $4$. 

 \begin{theorem} \label{th: except-11-dim}
The $11$-dimensional $3$-Sasakian homogeneous space $S={\mathrm{G}}_2 / {\mathrm{Sp}}(1)$
is formal.
\end{theorem}

\begin{proof}
The space $S={\mathrm{G}}_2 / {\mathrm{Sp}}(1)$ is the total space of the $\SO(3)$-bundle $\SO(3)\rightarrow \, S\, \to GI$
with Pontryagin class given by the integral cohomology class of the quaternionic K\"ahler $4$-form $\Omega$ on $GI$. This
must be (a non-zero multiple of) the class $x$ in $H^4(GI)$.
By Theorem \ref{th:1-conn-positive},  $GI={\mathrm{G}}_2\,/ {\mathrm{SO}}(4)$ is simply connected, and
so $\SO(3)\,\to \, S\, \to GI$ is a rational fibration with rational fiber $S^3$. 
Thus, according with section \ref{subsect:fibrations}, if 
$(\mathcal{A}, d_{\mathcal{A}})$ is a model of $GI$, then
$(\mathcal{A} \otimes \bigwedge(u), d)$, with $|u|\,=\,3$, $d\vert_{\mathcal{A}}\,=\,d_{\mathcal{A}}$
and $du\,=\,x$, is a model of $S$.
Furthermore,  $GI={\mathrm{G}}_2\,/ {\mathrm{SO}}(4)$ is formal because it is a symmetric space (see
also Theorem \ref{th:formalquat}). Hence a model of $GI$ is  the DGA
$(H^*(GI), 0)$. Then, a model of $S$ is the DGA
$(H^*(GI)\otimes \bigwedge(u), d)$, where $|u|\,=\,3$ and $du=x$. 

By \eqref{cohom:GI} the unique non-zero de Rham cohomology groups of $GI$ are 
$$
H^0(GI)=\langle 1\rangle\,, \quad H^4(GI)=\langle x \rangle\,,\quad H^8(GI)=\langle x^2\rangle\,.
$$
Therefore, the cohomology of $S$ is
 $$
 H^0(S)= 1,\, \quad  H^i(S) =0\,,  \quad 1\leq i \leq 10, \, \quad   H^{11}(S)= x^3 u.
 $$
Then, the minimal model of $S$ must be a differential graded algebra
$(\bigwedge V,d)$, where $V^j\,=\, 0$, and so $C^j\,=\, 0\,=\,N^j$, for $1\leq j\leq 10$. 
In particular, $N^j=0$ for $j\leq 5$. Thus, according with Definition~\ref{def:primera}, 
the manifold $S$ is $5$-formal and, by Theorem~\ref{fm2:criterio2}, $S$ is formal. 
\end{proof}

\subsection{The Wolf space $FI$}
Now we consider the $28$-dimensional homogeneous quaternionic K\"ahler manifold   
 $$
 {FI}\,=\,\frac{{\mathrm{F}}_4}{{\mathrm{Sp}}(3)\cdot{\mathrm{Sp}}(1)}.
 $$ 
Its rational cohomology is given by \cite{Ishitoya-Toda}
  \begin{equation}\label{cohom:FI} 
  H^*(FI) =\mathbb{Q}[x, y, z]/(x^3-12xy+8z, xz-3y^2, y^3-z^2),
 \end{equation}
where $| x|=4$, $|y|=8$ and $|z|=12$.

 \begin{theorem} \label{th: except-31-dim}
The $31$-dimensional $3$-Sasakian homogeneous space $S={\mathrm{F}}_4/{\mathrm{Sp}}(3)$
is formal.
\end{theorem}

\begin{proof}
The space $S={\mathrm{F}}_4/{\mathrm{Sp}}(3)$ is the total space of the $\SO(3)$-bundle $\SO(3)\,\to \, S\, \to FI$ with Pontryagin class 
given by the cohomology class of the quaternionic K\"ahler $4$-form which is (a non-zero multiple of) $x$.
As $FI={\mathrm{F}}_4\,/ ({\mathrm{Sp}}(3)\cdot{\mathrm{Sp}}(1))$ is simply connected and formal,
a model for $S$ is given by the DGA $(H^*(FI) \otimes \bigwedge(u), d)$ with $|u|\,=\,3$ and $du=x \in H^4(FI)$.

By \eqref{cohom:FI} the unique non-zero de Rham cohomology groups of $FI$ are 
\begin{align*}
 &H^0(FI)=\langle 1\rangle,  && H^4(FI)=\langle x \rangle, & &H^8(FI)=\langle x^2, y\rangle, 
 & &H^{12}(FI)=\langle x^3, x y\rangle, \\ &H^{16}(FI)=\langle x^4, y^2\rangle, & & H^{20}(FI)=\langle x^2 y, x y^2\rangle, &
&H^{24}(FI)=\langle x^6, x y^2\rangle, && H^{28}(FI)=\langle x^7\rangle. 
\end{align*}
Therefore, the unique non-zero cohomology groups of $S$ are
 $$
 H^0(S)= 1,\, \quad  H^8(S) =\langle y\rangle\,, \quad  H^{23}(S) =\langle PD(y)\rangle\,,\quad  H^{31}(S) =\langle x^7 u\rangle\,,
 $$
where $PD(y)$ is the Poincar\'e dual of $y$. Then the minimal model of $S$ must be a differential graded algebra
$(\bigwedge V,d)$, being $\bigwedge V$ the free algebra of the form
$\bigwedge V=\bigwedge(a_{8},v_{15})\otimes \bigwedge V^{\geq 16}$, where $|a_{8}|=8$, $|v_{15}|=15$, $da_{8}=0$ and
$dv_{15}=a_{8}^2$. According with Definition~\ref{def:primera}, $S$ is $14$-formal because 
$N^j=0$, for $j\leq 14$. Moreover, $S$ is 
$15$-formal. In fact, take $\alpha\in I(N^{\leq {15}})$ a closed
element in $\bigwedge V$. As $H^*(\bigwedge V)=H^*(S)$ has only non-zero cohomology in degrees $0, 8, 23$ and $31$, 
it must be $|\alpha|= 23, 31$. If $|\alpha|=23$ then $\alpha=a_{8} \,v_{15}$ which is not closed, and if
$|\alpha|= 31$ then $\alpha=a_{8}^2\,v_{15}$ which is not closed either.
So, according with Definition \ref{def:primera}, $S$ is $15$-formal and, by Theorem
\ref{fm2:criterio2}, $S$ is formal.
\end{proof}

\subsection{The Wolf space $EII$}
For the $40$-dimensional homogeneous quaternionic K\"ahler manifold 
 $$
 EII \,=\,\frac{{\mathrm{E}}_6}{{\mathrm{SU}}(6)\cdot{\mathrm{Sp}}(1)}\,,
 $$ 
we know that its rational cohomology is \cite{Ishitoya}
 \begin{equation}\label{cohom:EII}
  H^*(EII) =\mathbb{Q}[x, y, z, t]/(R_{12},R_{16},R_{18},R_{24}),
  \end{equation}
where $|x|=4$, $|y|=6$, $|z|=8$, $|t|=12$, and
\begin{align} \label{eqn:RR}
 & R_{12}=y^2-8t-6zx+x^3, & &
 R_{16}=x^4+12xt-6x^2z-3z^2,\\
 &R_{18}=yt, && R_{24}=t^2+z^3-\tfrac32xzt. \nonumber
\end{align}

\begin{theorem} \label{th: except-43-dim}
The $3$-Sasakian homogeneous space $S={\mathrm{E}}_6/{\mathrm{SU}}(6)$ of dimension $43$
is formal.
\end{theorem}

\begin{proof}
The space $S={\mathrm{E}}_6/{\mathrm{SU}}(6)$ is the total space of the $\SO(3)$-bundle $\SO(3)\,\to \, S\, \to EII$
 with Pontryagin class given by the quaternionic $4$-form which is (a non-zero multiple of) $x$.
 Since $EII={\mathrm{E}}_6\,/ ({\mathrm{SU}}(6)\cdot{\mathrm{Sp}}(1))$ is simply connected and 
formal, a model for $S$ is the DGA 
$(H^*(EII)\otimes \bigwedge(u), d)$, where $|u|\,=\,3$ and $du=x \in H^4(EII)$.

By \eqref{cohom:EII} and the relations given in (\ref{eqn:RR}), the $20$ first de Rham cohomology 
groups of $EII$ are 
 \begin{align*}
&H^0(EII)=\langle 1\rangle ,     
&&H^4(EII)=\langle x \rangle ,
&&H^6(EII)=\langle y\rangle ,   \\
&H^8(EII)=\langle x^2, z\rangle ,      
&&H^{10}(EII)=\langle x y\rangle ,         
&&H^{12}(EII)=\langle x^3, x z, y^2\rangle , \\
&H^{14}(EII)=\langle x^2 y, y z\rangle ,       
&&H^{16}(EII)=\langle x^4, x^2 z, x y^2\rangle ,  
&&H^{18}(EII)=\langle x^3 y, xyz\rangle ,     \\
&H^{20}(EII)=\langle x^5, x^3 z, x^2 y^2, y^2 z\rangle , \hspace{-5mm}
\end{align*}
and $H^{2i+1}(EII)\,=\,0$, for $0\leq i \leq 9$.
Therefore, the $21$ first de Rham cohomology groups of $S$ are
  \begin{align*}
  & H^0(S)= 1, && H^6(S) =\langle y\rangle, &&   H^8(S) =\langle z \rangle, \\
& H^{12}(S) =\langle y^2 \rangle, &&   H^{14}(S) =\langle y z\rangle,  && H^{20}(S) =\langle y^2 z \rangle,
\end{align*}
and $H^{2i+1}(S) = 0$, for $0 \leq i \leq 10$. 

Then the minimal model of $S$ must be a differential graded algebra
$(\bigwedge V,d)$, being $\bigwedge V=\bigwedge(a_{6}, a_{8},v_{15}, v_{17})\otimes \bigwedge V^{\geq 22}$, where $|a_{6}|=6$, $|a_{8}|=8$, 
$|v_{15}|=15$, $|v_{17}|=17$,
$d a_{6}=0=d a_{8}$, $dv_{15}=a_{8}^2$ and $dv_{17}=a_{6}^3$. 
According with Definition~\ref{def:primera}, we have 
$N^j=0$, for $j\leq 14$, thus the manifold $S$ is $14$-formal. Let us see that it is $21$-formal. 
For this, it is sufficient to prove that $S$ is $17$-formal because $N^j=0$, for $18 \leq j\leq 21$.
Take $\alpha\in I(N^{\leq {17}})$ a closed
element in $\bigwedge V$. As $H^*(\bigwedge V)=H^*(S)$ has only non-zero cohomology 
in degrees $0, 6, 8, 12, 14, 20, 23, 29, 31, 35, 37$ and $43$, 
it must be $|\alpha|=  23, 29, 31, 35, 37, 43$, that is 
$\alpha$ has odd degree. 
 If $|\alpha|=29$, then $\alpha=a_{6}\, a_{8}\,v_{15}$ which is not closed, and if 
$|\alpha|=31$, then $\alpha=a_{6}\, a_{8}\,v_{17}$ which is not closed either. In the other cases, $\alpha$ is of the form 
$\alpha=P_1\,v_{15}+P_2\,v_{17}$, where $P_1,P_2 \in \bigwedge(a_{6}, a_{8})$. Now, 
$d\alpha=0$ implies that $P_1=P \, a_{6}^3$ and $P_2=-P \, a_{8}^2$, for some $P\in \bigwedge (a_{6}, a_{8})$. So 
$\alpha=d(P\, v_{15}\,  v_{17})$ is exact. Thus any closed element in
the ideal $I(N^{\leq 21})$ is exact and hence $S$ is $21$-formal.
By Theorem \ref{fm2:criterio2}, the manifold $S$ is formal.
\end{proof}

\subsection{The Wolf space $EVI$}
Now we consider the $64$-dimensional homogeneous quaternionic K\"ahler manifold 
 $$
 {EVI}\,=\,\frac{E_7}{\mathrm{Spin}(12)\cdot \mathrm{Spin}(1)}.
 $$
Its rational cohomology ring is ~\cite{Nakagawa}
 \begin{equation}\label{cohom:EVI}
 H^*(EVI)=\mathbb{Q}[x,y,z]/(R_1,R_2,R_3),
 \end{equation}
where $|x|=4$, $|y|=8$, $|z|=12$, and 
\begin{align} \label{eqn:RRR}
R_1&=-2x^4y+2x^3z-3xyz+\tfrac18y^3+3z^2, \nonumber\\
R_2&=x^7-x^5y+4x^4z-\tfrac32x^3y^2-\tfrac38xy^3+3xz^2-\tfrac34y^2z,\\
R_3&=4x^7y+3x^5y^2+8x^4yz+x^3y^3+4x^3z^2+6x^2y^2z+\tfrac3{16}xy^4+\tfrac38y^3z+8z^3.\nonumber
\end{align}

\begin{theorem}\label{th: except-67-dim}
The $3$-Sasakian homogeneous space $S={\mathrm{E}}_7/{\mathrm{Spin}}(12)$ of
dimension $67$ is formal.
\end{theorem}

\begin{proof}
The space $S={\mathrm{E}}_7/{\mathrm{Spin}}(12)$ is the total space of the $\SO(3)$-bundle $\SO(3)\rightarrow S\rightarrow EVI$ 
with Pontryagin class given by the the quaternionic $4$-form which is (a non-zero multiple of) $x$. 
Since $EVI=E_7/(\mathrm{Spin}(12)\cdot \mathrm{Spin}(1))$ is simply connected and formal, a model of $S$ is given by 
$(H^\ast(EVI)\otimes\bigwedge(u),d)$, where 
$|u|=3$ and $du=x\in H^4(EVI)$. 

By \eqref{cohom:EVI} and the relations (\ref{eqn:RRR}), 
the unique non-zero de Rham cohomology
groups of $EVI$ up to degree~$33$ are
\begin{align*}
 & H^0(EVI)=\langle1\rangle, 
 && H^4(EVI)=\langle x\rangle,\\
 & H^{8}(EVI)=\langle x^2,y\rangle,
&& H^{12}(EVI)=\langle x^3,xy,z\rangle,\\
 &H^{16}(EVI)=\langle x^4,x^2y,y^2,xz\rangle, 
 &&H^{20}(EVI)=\langle x^5,x^3y,xy^2,x^2z,yz\rangle,\\
 &H^{24}(EVI)=\langle x^6,x^4y,x^2y^2,y^3,x^3z,xyz\rangle, 
 &&H^{28}(EVI)=\langle x^7,x^5y,x^3y^2,xy^3,x^4z,x^2yz\rangle,\\
 &H^{32}(EVI)=\langle x^8,x^6y,x^4y^2,x^2y^3,y^4,x^5z,x^3yz\rangle. \hspace{-5mm}
\end{align*}
Therefore, the $32$-first non-zero de Rham cohomology groups of $S$ are
\begin{align*}
 &H^0(S)=\langle1\rangle,&& H^8(S)=\langle y\rangle,&&H^{12}(S)=\langle z\rangle, 
 && H^{16}(S)=\langle y^2\rangle, \\
 &H^{20}(S)=\langle yz\rangle, && H^{24}(S)=\langle y^3\rangle,
 &&H^{32}(S)=\langle y^4\rangle,
\end{align*}
and $H^{2i+1}(S)=0$ for $0\leq i\leq 16$. The minimal model of $S$ must be the differential graded algebra 
$(\bigwedge V,d)$ where $\bigwedge V=\bigwedge(a_{8}, a_{12}, v_{23}, v_{27})\otimes \bigwedge V^{\geq 34}$, 
where $|a_{8}|=8$, $|a_{12}|=12$, $|v_{23}|=23$ and $|v_{27}|=27$, and the differential $d$ is given by $da_{i}=0$ $(i=8, 12)$, $dv_{23}=a_{8}^3+24a_{12}^2$ and
$dv_{27}=a_{8}^2\,a_{12}$. 
Thus, the manifold $S$ is $22$-formal because the space $N^j=0$, for $j\leq 22$. 
Moreover, $S$ is $33$-formal. 
Take  $\alpha\in I(N^{\leq 33})$.
Since $H^*(\bigwedge V)=H^*(S)$ has only non-zero cohomology in degrees $0, 8, 12, 16, 20, 24, 32, 35, 43, 47, 51, 55, 59$ and $67$, 
the degree of $\alpha$ must be $|\alpha|=  35, 43, 47, 51, 55, 59, 67$,  and so $\alpha$ has not component in $v_{23}\, 
v_{27}$. Therefore, $\alpha$
is of the form $\alpha=P_1 \,v_{23}+P_2\, v_{27}$,
where $P_1$ and $P_2$ live in the subalgebra $\bigwedge (a_{8}, a_{12})$ of 
$\bigwedge V^{\leq 33}$. The equality
$d\alpha=0$ implies $P_1=P\, dv_{27}$ and $P_2=-P\, dv_{23}$,
 for some $P\in \bigwedge (a_{8}, a_{12})$. Hence, $\alpha=d(P\, v_{23}\,  v_{27})$ which proves that
any closed element in the ideal $I(N^{\leq 33})$ is exact. By
Definition~\ref{def:primera}, $S$ is $33$-formal and, 
by Theorem~\ref{fm2:criterio2}, $S$ is formal.
\end{proof}

\subsection{The Wolf space $EIX$}
We consider the homogeneous quaternionic K\"ahler manifold 
 $$
 {EIX}\,=\,\frac{{\mathrm{E}}_8}{{\mathrm{E}}_7\cdot{\mathrm{Sp}}(1)}.
 $$
By \cite{Salamon2} (see also \cite{Piccinni}) its rational cohomology ring is
  \begin{equation}\label{cohom:EIX}
  H^*(EIX)=\mathbb{Q}[x_4, x_{12}, x_{20}]/(x_{12}^4, x_{20}^2),
  \end{equation}
where $|x_i|=i$, with $i= 4, 12, 20$. 

\begin{theorem}\label{th: except-67-dim}
The $3$-Sasakian homogeneous space $S={\mathrm{E}}_8\,/ {\mathrm{E}}_7$ of dimension 115 is formal.
\end{theorem}

\begin{proof}
The space $S={\mathrm{E}}_8\,/ {\mathrm{E}}_7$ is the total space of the $\SO(3)$-bundle $\SO(3)\,\to \, S\, \to EIX$
with Pontryagin class given by the cohomology class of the quaternionic $4$-form, which is
(a non-zero multiple of) $x_4$. Since 
$EIX={\mathrm{E}}_8 / ({\mathrm{E}}_7\cdot{\mathrm{Sp}}(1))$ is simply connected 
and formal, a model of $S$ 
is the differential algebra $(H^\ast(EIX)\otimes\bigwedge(u),d)$ with $|u|=3$ and $du=x_{4}$. 

Using the cohomology algebra in (\ref{cohom:EIX}) (which this time we do not write 
explicitly because it is very long, and easy to do for the reader) and the model for $S$, we 
get easily that the non-zero de Rham cohomology groups of $S$ up to degree $58$ are 
 \begin{align*}
& H^0(S)=\langle 1\rangle,  &&H^{12}(S)=\langle x_{12}  \rangle, && H^{20}(S)=\langle  x_{20} \rangle\, && H^{24}(S)=\langle x_{12}^2 \rangle,\\
&H^{32}(S)=\langle x_{12} x_{20} \rangle, && H^{36}(S)=\langle x_{12}^3 \rangle, && H^{44}(S)=\langle x_{12}^2 x_{20}\rangle\, && H^{56}(S)=\langle x_{12}^3 x_{20}\rangle.
\end{align*}
By Poincar\'e duality, there exist elements $x_{i} \in H^{i}(S)$, $ i = 59, 71, 79, 83, 91, 95,103, 115$, 
such that (see  \cite{Nakagawa2})
$$
x_{12}^3 x_{20}x_{59} = x_{12}^2 x_{20}x_{71} =  x_{12}^3 x_{79} = x_{12} x_{20}x_{83} = x_{12}^2 x_{91} = x_{20}x_{95} = x_{12}x_{103}= x_{115},
$$
and 
 \begin{align*}
 & x_{71} = x_{12} x_{59}\,, && x_{79} = x_{20}x_{59}\,, &&  x_{83} = x_{12}^2 x_{59}\,,\\
 & x_{91} = x_{12} x_{20} x_{59}\,, &&  x_{103} = x_{12}^2 x_{20} x_{59}\,, && x_{115} =x_{12}^3 x_{20} x_{59}.
\end{align*}

Then the minimal model of $S$ must be a differential graded algebra
$(\bigwedge V,d)$, with 
$\bigwedge V=\bigwedge(a_{12}, a_{20}, v_{39}, v_{47})\otimes \bigwedge V^{\geq 59}$, where $|a_i|=i$ for $i=12,20$, 
$|v_j|=j$ for $j=39,47$, and $d$ is given
by $d a_{12}=0=d a_{20}$, $dv_{39}=a_{20}^2$ and $dv_{47}=a_{12}^4$. 
According with Definition~\ref{def:primera}, the manifold $S$ is $38$-formal because $N^j=0$, for $j\leq 38$.
To prove that $S$ is $57$-formal it is sufficient to prove that $S$ is $47$-formal since $N^j=0$, for $48 \leq j\leq 57$.
Let $\alpha\in I(N^{\leq {47}})$ be a closed
element in $\bigwedge V$. Since $H^*(\bigwedge V)=H^*(S)$ has only non-zero cohomology in degrees 
$0, 12, 20, 24, 32, 36, 44, 56, 59, 71, 79, 83, 91, 95, 103$ and $115$, 
it must be $|\alpha|=  59, 71, 79, 83, 91, 95, 103, 115$. (Note that $|\alpha|\not=44, 56$ because $|v_{39}|=39$, $|v_{47}|=47$ and $b_{2i+1}(S)=0$, for $0\leq i \leq 28$.)
For each of these cases, $\alpha$ is given as follows:
$\alpha=\lambda_1\, a_{20}\,v_{39}+\mu_1\, a_{12}\,v_{47}$ if $|\alpha|=59$;
$\alpha=\lambda_2\, a_{12}\,a_{20}\,v_{39}+\mu_2\, a_{12}^2\,v_{47}$ if $|\alpha|=71$;
$\alpha=\lambda_3\, a_{20}^2\,v_{39}+\mu_3\, a_{12}\, a_{20}\,v_{47}$ if $|\alpha|=79$; 
$\alpha=\lambda_4\, a_{12}^2\, a_{20}\,v_{39}+\mu_4\, a_{12}^3\,v_{47}$ if $|\alpha|=83$;
$\alpha=\lambda_5\, a_{12}\,a_{20}^2\,v_{39}+\mu_5\, a_{12}^2\,a_{20}\,v_{47}$ if $|\alpha|=91$; 
$\alpha=\lambda_6\, a_{12}^3\,a_{20}\,v_{39}+\mu_6\,  a_{12}^4\,v_{47}$ if $|\alpha|=95$; 
$\alpha=\lambda_7\, a_{12}^2\,a_{20}^2\,v_{39}+\mu_7\,  a_{12}^3\,a_{20}\,v_{47}$ if $|\alpha|=103$; and
$\alpha=\lambda_8\, a_{12}^3\,a_{20}^2\,v_{39}+\mu_8\,  a_{12}^4\,a_{20}\,v_{47}$ if $|\alpha|=115$, where $\lambda_i, \mu_j\in \mathbb{R}$. 
One can check that $\alpha$ is not closed
in any of these cases. Then Definition \ref{def:primera} implies that the manifold $S$ is $47$-formal.
Hence $S$ is $57$-formal, and by Theorem \ref{fm2:criterio2}, $S$ is formal.
\end{proof}

\begin{proposition}
Let $F=\SU(2)$ or $\SO(3)$, and let $F\to S\to B$ be a principal fiber bundle, 
where $B=GI, FI, EII, EVI$ or $EIX$.
Then $S$ is formal. 
\end{proposition}
\begin{proof}
The principal fiber bundle $F\to S\to B$, with $B$ being one of the exceptional Wolf spaces
$B=GI, FI, EII, EVI$ or $EIX$, is such that in all these cases, $H^4(B)$ is one-dimensional, generated by some
$x\in H^4(B)$. Then the Euler class of $S$ is $e(S)=a\, x$. If $a=0$, the fibration is rationally a product and hence
$S$ is formal. If $a\neq 0$, then the fibration is rationally the same as the one considered in each of the previous
subsections. Therefore $S$ is formal as it has been computed above.
\end{proof}

\medskip

\section*{Acknowledgements} We would like to thank S. Salamon for useful suggestions.  The first and third authors
were partially supported by MINECO-FEDER Grant MTM2014-54804-P
and Gobierno Vasco Grant IT1094-16, Spain. 
The second author was partially
supported by MINECO-FEDER Grant (Spain) MTM2015-63612-P.

\end{document}